\documentclass[12 pt]{amsart}
\usepackage{a4wide}

\usepackage{amsmath,amssymb,amsfonts,amsthm}
\numberwithin{equation}{section}

\usepackage{color}
\usepackage{enumerate}

\newtheorem{theorem}{Theorem}[section] 
\newtheorem{proposition}[theorem]{Proposition} 
\newtheorem{lemma}[theorem]{Lemma}

\theoremstyle{definition}
 
\newtheorem{construction}[theorem]{Construction}
\newtheorem{notn}[theorem]{Notation}

\newcommand{\PGammaL}{\mathop{\mathrm{P}\Gamma\mathrm{L}}}

\newcommand{\GL}{\mathop{\mathrm{GL}}}

\newcommand{\la}{\langle}
\newcommand{\ra}{\rangle}
\DeclareMathOperator{\Aut}{Aut}
\DeclareMathOperator{\Inn}{Inn}
\DeclareMathOperator{\Out}{Out}
\newcommand{\norml}{\vartriangleleft}
\newcommand\Dl{\Delta}
\newcommand\cA{{\mathcal A}}
\newcommand\cB{{\mathcal B}}
\newcommand\cG{{\mathcal G}}

\DeclareMathOperator{\im}{Im}

\mathchardef\tnode="020E % temporary node
\def\arc{% unlabelled stroke for projective plane
  \hbox{\kern -0.15em
  \vbox{\hrule width 3.4em height 0.6ex depth -0.5 ex}
  \kern -0.33em}}

\def\darc{% double arc for GQ
  \rlap{\lower0.2ex\arc}{\raise0.2ex\arc}
  \kern -0.05em}

\def\stroke#1{% labelled stroke for diagrams
  \kern 0.05em
  \rlap\arc{{\textstyle{#1}}\atop\phantom\arc}
  \kern -0.30em}

\def\dstroke#1{% probably just for triple cover of Sp(4,2) quadrangle
  \kern 0.05em
  \rlap\darc{{\textstyle{#1}}\atop\phantom\darc}
  \kern -0.30em}

\def\etc{\>\>\>\>\cdots\>\>\>\>} % for ellipsis
\def\centerscript#1{% centres text over or under a node
  \setbox0=\hbox{$\tnode$}
  \hbox to \wd0{\hss$\scriptstyle{#1}$\hss}}

\def\node{%labelled node; usage \node or \node^{label} or \node_{label}
  \def\super{}
  \def\sub{}
  \futurelet\next\dolabellednode}

  \let\sp=^
  \let\sb=_

  \def\dolabellednode{%
    \ifx\next\sb\let\next\getsub
    \else
      \ifx\next\sp\let\next\getsuper
      \else\let\next\donode
      \fi
    \fi
    \next}

  \def\getsub_#1{\def\sub{#1}\futurelet\next\dolabellednode}
  \def\getsuper^#1{\def\super{#1}\futurelet\next\dolabellednode}

  \def\donode{%
    \rlap{$\mathop{\phantom\tnode}\limits_{\centerscript{\sub}}
    ^{\centerscript{\super}}$}\tnode}

\def\varcdn{%vertical arc with node at bottom - use as subscript to node
  \kern -0.03em\vbox{\kern -0.5ex
  \hbox to \wd0{\hss\vrule width 0.04em depth 5.8ex\hss}
  \kern -0.3ex
  \hbox{$\tnode$}}}

\mathchardef\tnodef="020F % temporary node

\def\nodef{%labelled filled
  \def\super{}
  \def\sub{}
  \futurelet\next\dolabellednodef}

  \let\sp=^
  \let\sb=_

  \def\dolabellednodef{%
    \ifx\next\sb\let\next\getsubf
    \else
      \ifx\next\sp\let\next\getsuperf
      \else\let\next\donodef
      \fi
    \fi
    \next}

  \def\getsubf_#1{\def\sub{#1}\futurelet\next\dolabellednodef}
  \def\getsuperf^#1{\def\super{#1}\futurelet\next\dolabellednodef}

  \def\donodef{%
    \rlap{$\mathop{\phantom\tnodef}\limits_{\centerscript{\sub}}
    ^{\centerscript{\super}}$}\tnodef}

\def\varcdnf{%vertical arc with filled node at bottom -
  \kern -0.03em\vbox{\kern -0.5ex
  \hbox to \wd0{\hss\vrule width 0.04em depth 5.8ex\hss}
  \kern -0.3ex
  \hbox{$\tnodef$}}}

\title[Amalgams related to $M_{24}$, $He$ and $A_{16}$]{A characterisation of weakly locally projective amalgams related to $A_{16}$ and the sporadic simple groups $M_{24}$ and $He$}
\author{Michael Giudici, A.~A.~Ivanov, Luke Morgan and Cheryl
  E.~Praeger}
  
\address{Michael Giudici, Luke Morgan, Cheryl E. Praeger: School of Mathematics and Statistics,
  University of Western Australia, 35 Stirling Highway, Crawley, WA 6009, Australia.} 
\email{michael.giudici@uwa.edu.au, luke.morgan@uwa.edu.au \newline \indent cheryl.praeger@uwa.edu.au}

\address{A. A. Ivanov: Department of Mathematics,
Imperial College, 180 Queen's Gate,  London, SW7 2BZ, United Kingdom }
\email{a.ivanov@imperial.ac.uk}

\date{}
\thanks{The first and third authors are supported by the  Australian  Research Council Discovery Project Number DP120100446 while the fourth author is supported by DP130100106.} 
\subjclass[2010]{20B25, 05C25, 20D08}

\begin{document}

\begin{abstract}
A simple undirected graph is weakly $G$-locally projective, for a group of automorphisms $G$, if for each vertex $x$, the stabiliser $G(x)$ induces  on the set of vertices adjacent to $x$  a doubly transitive action  with socle the projective group $L_{n_x}(q_x)$ for an integer $n_x$ and a prime power $q_x$.  It is $G$-locally projective if in addition $G$ is vertex transitive. A theorem of Trofimov reduces the classification of the $G$-locally projective graphs to the case where the distance factors are as in one of the known examples. Although an analogue of Trofimov's result is not yet available for weakly locally projective graphs, we would like to begin a program of characterising some of the remarkable examples. We show that if a graph is weakly locally projective with each $q_x =2$ and $n_x = 2$ or $3$, and if the distance factors are as in the examples arising from the rank 3 tilde geometries of the groups $M_{24}$ and $He$, then up to isomorphism there are exactly two possible amalgams. Moreover, we consider an infinite family of amalgams of type $\mathcal{U}_n$ (where each $q_x=2$ and $n=n_x+1\geq 4$) and prove that if $n\geq 5$ there is a unique amalgam of type $\mathcal{U}_n$ and it is unfaithful, whereas if $n=4$ then there are exactly four amalgams of type $\mathcal{U}_4$, precisely two of which are faithful, namely the ones related to $M_{24}$ and $He$, and one other which has faithful completion $A_{16}$.
\end{abstract}
\maketitle

\section{Introduction}
Let $\Dl$ be an undirected, connected locally finite graph and $G$ be an 
automorphism group of $\Dl$. For a vertex $x$ of $\Dl$ let $G(x)$ be the 
stabilizer of $x$ in $G$ and $\Dl(x)$ be the set of neighbours of $x$
in $\Dl$. Let $G(x)^{\Dl(x)}$ denote the permutation group induced by
$G(x)$ on $\Dl(x)$.
We say that the action of $G$ on $\Dl$ is {\em weakly locally projective} if 
the following holds: for every vertex $x \in \Dl$ there is a positive 
integer $n_x$ and a prime power $q_x$ such that 
$$|\Dl(x)|=(q_x^{n_x}-1)/(q_x-1);$$
$$L_{n_x}(q_x) \leqslant G(x)^{\Dl(x)} \leqslant \PGammaL(n_x,q_x).$$
Here $L_{n_x}(q_x)$ is considered as a doubly transitive permutation group 
of degree $|\Dl(x)|$ and $\PGammaL(n_x,q_x)$ is the normalizer of 
$L_{n_x}(q_x)$ in the symmetric group on $\Dl(x)$. If a weakly locally 
projective action is also vertex-transitive it is said to be 
{\em locally projective}. Since every weakly locally projective action is 
edge-transitive it is easy to see that either

\begin{itemize}

\item[{\rm (a)}] the action is locally projective and there is a pair 
$(n,q)$ such that $n_x=n$ and $q_x=q$ for every $x \in \Dl$, or

\item[{\rm (b)}] there is bipartition $\Dl=\Dl_1 \cup \Dl_2$ of 
$\Dl$ and a quadruple of parameters $(n_1,q_1;n_2,q_2)$ 
such that $n_x=n_i$, $q_x=q_i$ whenever $x \in \Dl_i$.

\end{itemize}

The sequence $(n,q)$ or $(n_1,q_1;n_2,q_2)$ is said to be 
the {\em type} of a locally or weakly locally projective action, respectively.

Suppose that $G$ acts (weakly) locally projectively on $\Dl$ and $\{x,w\}$ is 
an edge of $\Dl$. Then the amalgam $\cA=\{G(x),G(w)\}$ formed by the 
stabilizers of the vertices incident to the edge is called a {\em (weakly) 
locally projective amalgam}. The {\em shape} of $\cA$ is the type of
the action of $G$ on $\Dl$. Since $G$ is an automorphism 
group of $\Dl$, it acts faithfully on the set of edges. Thus
$G(x)\cap G(w)=G(x,w)$ does not contain a nontrivial subgroup which is normal
in both $G(x)$ and $G(w)$. Amalgams with this property are called faithful. 

The present paper makes a modest contribution to the classification of the 
weakly locally projective amalgams. We were motivated by the following 
geometrical constructions.

\begin{construction}
\label{con:1}
Let $\cG$ be a geometry with a diagram
$$\node_1^1\stroke{{\rm X}}\node_2^2\arc\node_2^3\etc\node^{n-1}_2\arc\node^n_2,$$
for some $X$ (see for example \cite[p 2]{IS02} for notation), let $G$
be a flag-transitive automorphism group of $\cG$ such that the
stabilizer of an element of type 1 induces the full automorphism group
$L_n(2)$ of the corresponding residue. Let $\Dl$ be a graph whose
vertices are the elements of type 1 in $\cG$ and two such vertices are
adjacent in $\Dl$ if in $\cG$ they are incident to a common element of
type 2. Then the action of $G$ on $\Dl$ is locally projective of type
$(n,2)$.  
\end{construction}

\begin{construction}
\label{con:2}
Let $\cG$ be a geometry with a diagram 
$$\node_2^1\stroke{{\rm Y}}\node_2^2\arc\node_2^3\etc\node^{n-1}_2\arc\node^n_2,$$
for some $Y$, let $G$ be a flag-transitive automorphism group of $\cG$
such that the stabilizer of an element of type 1 induces the full
automorphism group $L_n(2)$ of the corresponding residue and the
stabilizer of an element of type 2 induces the group $S_3 \cong
L_2(2)$ on the set of elements of type 1 incident to that element. Let
$\Dl$ be a graph whose vertices are the elements of type 1 and 2 in
$\cG$ and two vertices are adjacent if they are incident as elements
of $\cG$. Then $\Dl$ is bipartite and the action of $G$ on $\Dl$ is
weakly locally projective of type $(n,2;2,2)$.  
\end{construction}

Remarkable locally projective amalgams can be obtained by Construction 
\ref{con:1} with $X$ being the geometry of edges and vertices of the Petersen graph.
The examples include geometries of the Mathieu groups $M_{22}$, $M_{23}$,
the Conway group $Co_2$, the Janko group $J_4$ and the Baby Monster
group, (see \cite{sashabook,IS02}).
 
Similarly remarkable weakly locally projective amalgams can be 
obtained by Construction \ref{con:2} with $Y$ being the rank 2 tilde geometry 
(the triple cover of the generalized quadrangle of order $(2,2)$ 
associated with $3 \cdot S_6$). The examples include geometries of the
Mathieu group $M_{24}$, the Conway group $Co_1$ and the Monster group $M$.

On the other hand, the classification of certain locally and weakly locally 
projective amalgams restricts the 
possibilities for the rank 2 residues $X$ and $Y$ in geometries as in 
Constructions \ref{con:1} and \ref{con:2}.

\medskip

Building on work of Tutte \cite{Tutte47}, in 1980 Djokovi\'c and Miller \cite{DM80} classified the locally projective amalgams of type $(2,2)$. In the same year the weakly locally projective 
amalgams of type $(2,2;2,2)$ were classified by Goldschmidt \cite{Goldschmidt80}.

In 1991 Trofimov \cite{trof1} proved a fundamental result on locally projective 
amalgams. The main consequence of Trofimov's Theorem 
is a bound on the order of the vertex 
stabilizer $G(x)$ in a locally projective action of type $(n,q)$ by 
a function of $n$ and $q$. Furthermore Trofimov determined 
the possibilities for the so-called {\em distance factors} 
$G_i(x)/G_{i+1}(x)$ (here $G_i(x)$ is the stabilizer in $G(x)$ of all the 
vertices whose distance from $x$ is at most $i$). All the possibilities 
for the distance factors are realized in known examples (see \cite{trofdurham}). 

Thus Trofimov's theorem reduces the classification problem of locally 
projective amalgams to its `restricted' version when the distance 
factors $G_i(x)/G_{i+1}(x)$ 
are assumed to be as in one of the known examples. 

The restricted problem was solved in 2004 by Ivanov and Shpectorov in
\cite{ivanshpect} for the amalgams obtained by Construction
\ref{con:1}. It is
worth mentioning that a few new amalgams were found within this
project whose distance factors are exactly as in the examples known
before. As a consequence of the classification it was shown that the
residue $X$ possesses a covering onto the geometry of edges and
vertices of one of the following three graphs: the complete graph
$K_4$ on 4 vertices;  the complete bipartite graph $K_{3,3}$ on 6
vertices; the Petersen graph.  

The analogue of Trofimov's theorem for weakly locally projective actions 
is not yet available. Nevertheless we would like to solve the
restricted problem for such amalgams coming from Construction \ref{con:2}. In
this paper we assume that $\Dl$ is a graph, $G$ is an automorphism
group of $\Dl$ whose action on $\Dl$ is weakly locally projective of
type $(3,2;2,2)$ and if $x$ is a vertex of valency 7 in $\Dl$ then the
distance factors are as follows: 
\begin{eqnarray}
G(x)/G_1(x) & \cong &L_3(2);\nonumber \\
G_1(x)/G_2(x)& \cong &C_2^3;\nonumber\\
G_2(x)/G_3(x)& \cong & C_2^3;\nonumber\\
G_3(x)/G_4(x)& \cong & 1; \label{distfactors}\\
G_4(x)/G_5(x)& \cong &C_2; \nonumber\\
G_5(x)& \cong &1.\nonumber
\end{eqnarray} 
This pattern appears in graphs obtained via Construction \ref{con:2} from the 
rank 3 tilde geometries of the groups $M_{24}$ and $He$. Our main
result is the following theorem.
\begin{theorem}
\label{thm:main}
Let $\Delta$ be a connected graph with automorphism group $G$ whose action is
weakly locally projective of type $(3,2;2,2)$. Let $x$ and $w$ be
adjacent vertices such that $x$ has valency 7 and $w$ has valency
3. Moreover, suppose that the distance factors are as in
$(\ref{distfactors})$. Then the following hold:
\begin{itemize}
\item[(1)] \label{thm:main1} $G(x)\cong 2^{1+6}_+\rtimes L_3(2)$, the centraliser of a $2A$-involution in $L_5(2)$;
\item[(2)] $G(w)=N\rtimes K$ where $N\cong C_2^4$ and 
$K\cong C_2^4\rtimes(S_3\times S_3)$, the stabiliser in $L_4(2)$ of a
  2-dimensional subspace of $N$;
\item[(3)] $G(x)\cap G(w)\cong 2^{1+6}_+\rtimes S_4$;
\item[(4)] Up to isomorphism, there are two faithful amalgams
  $\cA=\{G(x),G(w)\}$. 
\end{itemize}
\end{theorem}

We see in Theorem \ref{thm:4} that
given $G_1\cong 2^{1+6}_+\rtimes L_3(2)$ and $G_2\cong N\rtimes K$
such that $G_1\cap G_2\cong 2^{1+6}_+\rtimes S_4$, there are four
amalgams $\{G_1,G_2\}$. However, only two arise from weakly locally projective
actions as in the remaining two, $G_1 \cap G_2 $ contains a nontrivial
subgroup which is normal in both $G_1$ and $G_2$. For one of the two faithful
amalgams  the sporadic groups $M_{24}$ and $He$ are completions, for the other $A_{16}$ is a completion, see Section~\ref{sec:realisations} for details. In fact, Theorem 4.7 applies to the infinite sequence of amalgams $(\mathcal U_n)_{n\geqslant 4}$ and we see that each member of this sequence is a unique (unfaithful) amalgam, unless $n=4$. This highlights how remarkable the amalgams related to the Held and Mathieu groups are.

\section{Distance two graphs and triangles}

Since we are studying the structure of the vertex stabilisers, as
opposed to the structure of $\Delta$, we may assume that $\Delta$ is a tree \cite[Chapter 1,\S4]{serre}.
Moreover, we do our analysis in the
\emph{distance two graph} of $\Delta$, that is, the graph with the
same vertex set as $\Delta$ but where two vertices are adjacent if and only if they are at
distance two in $\Delta$. Since $\Delta$ is connected and
bipartite, the distance two graph of $\Delta$ has two connected
components, one containing the vertices of $\Delta$ of valency 7 and
the other containing the vertices of $\Delta$ of valency 3. 
 For a graph $\Sigma$ with vertex $v$, we denote
the set of vertices at distance $i$ from $v$ by $\Sigma_i(v)$ and  write $\Sigma(v)$ for $\Sigma_1(v)$. 
For vertices $u$ and $v$ we will write $u \sim v$ to indicate that $u \in \Sigma(v)$. If a group $R$ acts on $\Sigma$ and $L$ is a subgroup of $R$ which stabilises a set $\Pi$ of vertices of $\Sigma$, we write $L^\Pi$ for the permutation group induced on $\Pi$ by $L$. Typically we use this notation where $L$ fixes a vertex $u$ and $\Pi=\Sigma(u)$.

We have the following lemma.

\begin{lemma}
\label{delta lemma}
Let $\Delta$ be a connected tree with automorphism group $G$ whose
action is weakly locally projective of type $(3,2;2,2)$ and suppose
that the distance factors are as in $(\ref{distfactors})$. Let
$\Gamma$ be the connected component of the distance two graph of
$\Delta$ which contains all the vertices of valency 7. Let $x\in V\Gamma$ and let $Q_i(x)$ be the stabiliser
in $G(x)$ of all vertices of $\Gamma$ of distance at most $i$ from $x$.  The 
following all hold.
\begin{enumerate}
\item[$(A1)$] $\Gamma$ is a connected graph of valency 14.
\item[$(A2)$] Every edge of $\Gamma$ is in a unique triangle.
\item[$(A3)$] $G$ is an arc-transitive automorphism group of $\Gamma$,
\item[$(A4)$] $G(x)$ induces $L_3(2)$ on the set of
  seven triangles in $\Gamma$ containing $x$.
\item[$(A5)$] For a triangle $T$ of $\Gamma$, the setwise stabiliser in $G$ of $T$ induces $S_3$ on $T$.
\item[$(A6)$] The kernel of $G(x)^{\Gamma(x)}$ acting on the set
  of triangles containing $x$ is $C_2^3$.
\item[$(A7)$] The distance factors for $\Gamma$ are as follows: 
\begin{eqnarray}
G(x)/Q_1(x) & \cong &C_2^3.L_3(2);\nonumber \\
Q_1(x)/Q_2(x)& \cong &C_2^3;\nonumber\\
Q_2(x)/ Q_3(x)& \cong & C_2 ;\nonumber\\
Q_3(x)& \cong &1.\nonumber
\end{eqnarray} 
\end{enumerate}
\end{lemma}
\begin{proof}
Since $\Delta$ is a tree, (A1) holds and given two vertices $x$, $y$ of valency 7 in
$\Delta$ that are at distance two there is a unique vertex $z$ of $\Delta$ at
distance two from both $x$ and $y$. Thus (A2) holds. Since for each vertex $v$ of $\Delta$ we have that
$G(v)$ is 2-transitive on $\Delta(v)$, it follows that $G(x)$ acts
transitively on the set of paths of length two starting at
$x$. Moreover, $G$ is a vertex-transitive automorphism group of
$\Gamma$ and since arcs in $\Gamma$ correspond to paths of length two
in $\Delta$ with initial vertex having valency 7 it follows that $G$
is an arc-transitive group of automorphisms of $\Gamma$. Thus (A3)
holds.

For a given vertex $x\in V\Gamma$, the action of $G(x)$ on the set of
triangles containing $x$ is equivalent to 
$G(x)^{\Delta(x)}\cong G(x)/G_1(x)$ and so by
(1.1), (A4) holds.  Moreover, since $G(v)^{\Delta(v)}\cong L_2(2)\cong S_3$ for each vertex $v$ of $\Delta$
of valency three (A5) holds. For a positive integer $i$ we have 
$\Gamma_i(x)=\Delta_{2i}(x)$ and $Q_i(x)=G_{2i}(x)$. Hence the kernel of
$G(x)^{\Gamma(x)}$ acting on the set of triangles containing $x$ is
$G_1(x)/G_2(x)$ and so (A6) holds. We obtain (A7) from (\ref{distfactors}).
\end{proof}

Note that each triangle in $\Gamma$ corresponds to a vertex of valency
3 in $\Delta$. Thus the problem of identifying the isomorphism type of
$\{G(x),G(w)\}$ for a given pair $x$, $w$ of adjacent vertices in
$\Delta$ is equivalent to identifying the isomorphism type of
$\mathcal{A}=\{G(x),G\{T\}\}$ where $G\{T\}$ is the setwise stabiliser in $G$ of the
 triangle $T=\{x,y,z\}$ in $\Gamma$. Note in particular that $\mathcal A$ is faithful since 
 $\langle G(x), G\{T\} \rangle = G$ by connectivity.

\section{Determining the structure of $G(x)$ and $G\{T\}$}
\label{sec:determine}

Throughout this section we assume the following hypothesis:
\begin{center}
$\Gamma$ is a connected graph on which $G$ acts faithfully such that (A1)--(A7) of Lemma~\ref{delta lemma} hold.
\end{center}

We first establish some notation for the action of $G$ on $\Gamma$ which will hold for the rest of the
paper. We let $Q(x)$ denote the kernel
of the action of $G(x)$ on the
set of triangles containing $x$. Then $Q(x)=G_1(x)=\mathrm O_2(G(x))$ and
$G(x)/Q(x)\cong L_3(2)$. As in Lemma \ref{delta lemma} we write $Q_i(x)$ for the stabiliser in $G(x)$ of all
vertices of distance at most $i$ from $x$ in $\Gamma$. We fix a triangle $T=\{x,y,z\}$ containing $x$ and write $G\{T\}$
for the setwise stabiliser in $G$ of $T$. We write $G(T)$ for the pointwise stabiliser of $T$ in $G$, so that 
$$G(T) = G(x) \cap G(y) \cap G(z)$$
and $G(T)$ is a normal subgroup of $G\{T\}$. Moreover $G\{T\} / G(T) \cong S_3$. Some more
normal subgroups of $G\{T\}$ which we will need are
\begin{eqnarray*}
F & = & \mathrm O_2(G\{T\}), \\
N  & = & Q(x) \cap Q(y) \cap Q(z).
\end{eqnarray*} 
 Since $G$ is both vertex and triangle transitive, statements proved about $G(x)$ and $G\{T\}$ apply to
arbitrary vertex and triangle stabilisers, and we will use  appropriate notation for subgroups
conjugate to named subgroups of $G(x)$ and $G\{T\}$.

We now collate information about the actions of various subgroups of
$G(x)$.  We have the following lemma since $\{ G(x), G\{T\} \}$ is a faithful amalgam.

\begin{lemma}
\label{lem:lem1}
Let $R$ be a subgroup of $G(x) \cap G\{T\}$ and suppose that $R$ is normal in both $G(x)$ and $G\{T\}$. Then $R=1$.
\end{lemma}

\begin{lemma}
\label{lem:fact of gxt}
We have $G(x) \cap G\{T\} = Q(x) G(T)$.
\end{lemma}
\begin{proof}
Suppose that $Q(x) \leqslant G(T)$. The normality of $Q(x)$ in $G(x)$ now implies that $Q(x)$ fixes $\Gamma(x)$ pointwise which gives $Q(x) = Q_1(x)$, a contradiction to (A7). Hence $Q(x) G(T) > G(T)$. By (A5), $G\{T\}$ is 2-transitive on $T$ and therefore contains an element fixing $x$ and interchanging $y$ and $z$. Hence 
$$ |G(x) \cap G\{T\} : G(T)| = 2$$ and the result follows.
\end{proof}

By (A7) we have $|Q_2(x)|=2$. For each $u \in V\Gamma$ let $e_u \in G(u)$
 be such that $Q_2(u)=\la e_u\ra$. Put
\begin{eqnarray*}
E_x & = & \la e_u\mid u\in\Gamma(x)\ra, \\
E_T & = & \la e_x, e_y, e_z \ra.
\end{eqnarray*}
Observe that $E_x$ is a normal subgroup of $G(x)$ contained in $Q_1(x)$ and that $E_T$ is a normal subgroup of $G\{T\}$ contained in $G(T)$. Above we mentioned that we will use similar notation for subgroups conjugate to $E_x$, $E_T$, etc. As an example of this, we write
$$E_y = \la e_v \mid v \in \Gamma (y) \ra. $$

\begin{lemma}
\label{lem:distinct}
The involutions $e_x$, $e_y$  and $e_z$ are all distinct.
\end{lemma}
\begin{proof}
Since $G\{T\}$ acts on $T$ primitively, if the assertion fails,
$e_x=e_y=e_z$. Then $Q_2(x) = Q_2(y) = Q_2(z)$ is normalised by $G(x)$ and by $G\{T\}$, a contradiction to Lemma \ref{lem:lem1}.
\end{proof}

\begin{lemma}
\label{lem:e_x} The following hold.
\begin{itemize}
\item[(1)] $[e_u,e_v]=1$ for all $u,v \in \Gamma$ with $d(u,v)\leqslant 2$.
\item[(2)] We have  $e_xe_ye_z=1$.
\item[(3)] The involutions $e_u$ for $u\in\Gamma(x)$ are
  pairwise distinct. 
\item[(4)] $E_x=Q_1(x)\cong C_2^4$ and $E_x=\langle e_x\rangle \cup\{e_u\mid u\in\Gamma(x)\}$.
\item[(5)] The action of $G(x)$ on the nontrivial elements of $E_x/\langle e_x\rangle$ is equivalent to the action of $G(x)$ on
the set of triangles containing $x$.
\end{itemize}
\end{lemma}
\begin{proof}
Let $u$ and $v$ be as in (1). By definition, we have $\la e_u \ra = Q_2(u)$ so that $e_u \in G(v)$.  Since $Q_2(v)$ is a normal subgroup of $G(v)$ of order two, we have $e_v \in \mathrm Z(G(v))$ whence $[e_u,e_v] = 1$. Thus (1) holds.

We now define an equivalence relation on the set $\Gamma(x)$ which will aid us in proving (2)--(5). For $u,v\in \Gamma(x)$ we say $u \approx v$ if and only if $\la e_x,e_u \ra = \la e_x, e_v \ra$. It is immediate that $\approx$ is an equivalence relation. Since $e_x \in \mathrm Z(G(x))$ and $G(x)$ preserves the set $\Gamma(x)$ we see that $\approx$ is a $G(x)$-invariant relation. If $\approx$ is the universal relation, we have that 
$$\la E_x, e_x \ra = \la e_x, e_y\ra = \la e_x, e_y, e_z\ra = E_T$$
and so $E_T$ is a normal subgroup of $G(x)$ and of $G\{T\}$, a contradiction to Lemma~\ref{lem:lem1}. 

Suppose now that $\approx$ is the trivial relation. Then for all $u,v \in \Gamma(x)$ we have $e_u \neq e_v$, so that 
$$|E_x| \geqslant | \{ e_ u \mid u \in \Gamma(x) \}| = 14.$$
 By (1)  $E_x$ is an elementary abelian 2-group and by definition, $E_x \leqslant Q_1(x)$. Since $|Q_1(x)|=2^4$ by (A7) we have that $E_x = Q_1(x)$ and since $e_x \in Q_1(x)$
$$E_x = \{1, e_x \} \cup \{e_u \mid u \in \Gamma(x) \}.$$
On the other hand, $| \la e_x, e_y \ra | = 2^2$ and $d=e_xe_y$ is distinct from $1$, $e_x$ and $e_y$.  Now $d \in E_x$ and therefore $d=e_f$ for some $f\in \Gamma(x)$. This implies $y \approx f$, a contradiction to our assumption that $\approx$ is trivial.

Suppose now that the blocks of $\approx$ have size seven. Since $G(x)$ is transitive on the triangles which contain $x$, it must be that the two blocks divide each triangle into two and for $u \in \Gamma(x)$ one of $u \approx y$ or $u \approx z$ holds. This means that $\la e_x, e_u \mid u \in \Gamma(x) \ra = \la e_x, e_y, e_z \ra = E_T$, a contradiction to Lemma~\ref{lem:lem1}. Hence the blocks for $\approx$ have size two. Let $B$ be a block containing $y$. Since $|G(x) : G(x)_B|=7$ we see that $Q(x) \leqslant G(x)_B$ and since $G(T)$ fixes $y$ we have $G(T) \leqslant G(x)_B$. Now $G(x)_B \geqslant Q(x)G(T)$ and so Lemma~\ref{lem:fact of gxt} shows that  $G(x)_B = G(x) \cap G\{T\}$. Hence the relation $\approx$ is the same as the relation ``in a triangle". Thus $y \approx z$ and by Lemma~\ref{lem:distinct} we have $e_x e_y = e_z$, which is (2).

Now if $u,v \in \Gamma(x)$ are such that $e_u = e_v$ then $u \approx v$ which means that $u$ and $v$ are in some triangle, $S$ say. Since there is $g\in G(x)$ with $S^g = T$ this means $e_y = e_z$, a contradiction to Lemma~\ref{lem:distinct}. Thus (3) holds. 

As argued above, it follows immediately from (3) that (4) holds. For (5) we let $\mathcal T$ be the set of triangles containing $x$. By (4) for each $u \in \Gamma(x)$ we let $u'$ be the unique vertex in $\Gamma(x)$ distinct from $u$ such that $u \approx u'$. Then we may define $\phi : E_x/\la e_x \ra^\# \rightarrow \mathcal T$
by $$ \phi : \la e_u, e_x \ra / \la e_x \ra \mapsto \{ x, u, u' \}.$$
Since $G(x)$ preserves the set of triangles and $\la e_u, e_x \ra = \la e_{u'},e_x \ra$ for all $u\in \Gamma(x)$ it follows that $\phi$ is a well defined $G(x)$-invariant map.
\end{proof}

In the next lemma, the natural homomorphism $\alpha:G(x)\rightarrow\Aut(E_x)$ is the homomorphism induced by
the conjugation action of $G(x)$ on $E_x$.

\begin{lemma}
\label{lem:dualaction}
Let $\alpha:G(x)\rightarrow \Aut(E_x)$ be the natural
homomorphism. Then
\begin{itemize}
\item[(1)] $\ker(\alpha)=E_x$;
\item[(2)] $\im (\alpha)$ is the stabiliser in $GL_4(2)$ of the 1-space $\langle e_x\rangle$ and
$\alpha(Q(x))$ is the group of transvections of $E_x$ with axis $e_x$;
\item[(3)] $E_x/\la e_x\ra= Q_1(x)/Q_2(x)$ and 
$Q(x)/Q_1(x)\cong Im(\alpha)$ are dual as modules for 
$G(x)/Q(x)\cong L_3(2)$;
\item[(4)] $Q(x)$ is extraspecial of plus type and with centre 
$Q_2(x)=\la e_x\ra$.
\end{itemize}
\end{lemma}
\begin{proof}
Since $E_x$ is abelian we have $E_x\leqslant \ker(\alpha)$.  As the action of
$G(x)$ on $E_x\backslash\{1\}$ is equivalent to its action on
$\Gamma(x)\cup \{x\}$, it follows
that $\ker(\alpha)=Q_1(x)=E_x$.  Hence (1) holds. Moreover, $\langle e_x\rangle=Q_2(x)\norml G(x)$ and so
$\im (\alpha)$ is contained in the stabiliser $S$ in $\GL_4(2)$ of $\langle e_x\rangle$. Now $S=C_2^3\rtimes
L_3(2)$ and using the First Isomorphism Theorem, $\im (\alpha)=S$. Since $Q(x)$ acts
trivially  on the set of triangles containing $x$, $Q(x)$ centralises
$E_x/\la e_x\ra$ by Lemma \ref{lem:e_x}(5). The kernel
of $S$ on $E_x/\la e_x \ra$ is $C_2^3$, which gives (2). For (3) a straightforward computation shows that
the actions of $G(x)/Q(x)$ on $E_x/\la e_x\ra$ and $Q(x)/E_x \cong \alpha(Q(x))$ are dual. 

Now $Q(x)$ is nonabelian as $Q(x)$ acts nontrivially on $\Gamma(x)$ and hence on $E_x$.   Moreover,
$E_x/\langle e_x\rangle$ is a normal subgroup of order $2^3$ of $Q(x)/\langle e_x\rangle$ and $L_3(2)$ acts
as a group of automorphisms of $Q(x)/\langle e_x\rangle$ so that it acts dually on $Q(x)/E_x$ and
$E_x/\langle e_x\rangle$. Thus by \cite[Lemma 3.4]{triextraspecial}, $Q(x)/\langle e_x\rangle$ is elementary
abelian. Now $Z(Q(x))$ is contained in $E_x$ and the action of $L_3(2)$ tells us that $Z(Q(x))=\langle
e_x\rangle$ or $E_x$. Since $Q(x)$ does not centralise $E_x$ we have that $Z(Q(x))=\langle e_x\rangle$.
Hence  $Q(x)$ is extraspecial and since $Q(x)$ contains the elementary abelian subgroup $E_x\cong C_2^4$, it
follows that $Q(x)$ is of plus type.
\end{proof}

At this stage we can say that $G(x)$ is an extension of $2_+^{1+6}$ by
$L_3(2)$. To determine the isomorphism type of $G(x)$ we need to
determine the extension involved.

\begin{lemma}
\begin{itemize}
\label{lem:ET}
\item[(1)] $G(T)$ induces an irreducible action of $S_3$ on 
$E_x/E_T \cong C_2^2$.
\item[(2)] $E_y\cap Q(x)=E_T$.
\end{itemize}
\end{lemma}
\begin{proof}
Part (1) follows from Lemma \ref{lem:dualaction}(3). 
By definition we have $E_T\leqslant E_y$ and so
$E_T\leqslant E_y\cap Q(x)$. Suppose equality does not hold and recall that $E_y \cong C_2^4$ and $E_T \cong C_2^2$. Since
$G(T)$ normalises $E_y\cap Q(x)$ and $E_y \neq E_T$, part (1) implies that 
 $E_y = E_y \cap Q(x)$. By symmetry we have
$E_x\leqslant Q(y)$. Thus $E_xE_y\leqslant Q(x)\cap Q(y)$. Since
$Q(x)$ and $Q(y)$ are extraspecial with 
derived subgroups $\la e_x\ra$ and $\la e_y\ra$ respectively, we have
$[E_xE_y,E_xE_y]\leqslant \la e_x\ra \cap \la e_y\ra=1$. Thus
$E_xE_y$ is an abelian subgroup of $Q(x)$. However, $E_x$ is a
maximal elementary abelian subgroup of $Q(x)$ and so $E_xE_y=E_x$
and hence $E_y=E_x$. This contradicts Lemma \ref{lem:lem1} for
$Q_1(x)$. Thus $E_y\cap Q(x)=E_T$.
\end{proof}

\begin{lemma}
\label{lem:N}
The following hold:
\begin{itemize}
\item[(1)] $N = Q(x) \cap Q(y) = Q(x) \cap Q(z) = Q(y) \cap Q(z)$ and $|Q(x) \cap G(T)|=2^6$;
\item[(2)] $|N| = 2^4$ and $E_T < N$;
\item[(3)] $N$ is a maximal elementary abelian subgroup of $Q(x)$.
\end{itemize}
\end{lemma}
\begin{proof}
Since $Q(x)$ acts transitively on $T - \{x\}$ there is $t\in Q(x)$ interchanging $y$ and $z$.  Since $Q(x) / \la e_x \ra$ is abelian, every subgroup of $Q(x)$ that contains $\la e_x \ra$ is normal in $Q(x)$. 
We have $e_x \in Q(x) \cap Q(y)$, hence 
$$Q(x) \cap Q(y) = (Q(x) \cap Q(y))^t = Q(x) \cap Q(z).$$
The same argument shows that $Q(y) \cap Q(z) = Q(x) \cap Q(z)$.

Since $x^{Q(y)}=\{x,z\}$ is of size two, it follows that 
$$|Q(y) \cap G(T)|=|Q(y)\cap G(x)|=2^6.$$
and so (1) holds.

Since $G(x) \cap G\{T\} = Q(x) G(T)$ we see that both $E_yQ(x)$ and $(Q(y) \cap G(x)) Q(x)$ are normal $2$-subgroups of $G(x) \cap G\{T\}$. Since $(G(x) \cap G\{T\})/Q(x) \cong S_4$ we have $E_y Q(x) = (Q(y)\cap G(x))Q(x)$ and
$$|E_yQ(x)/Q(x)| = 2^2 = |Q(y) \cap G(x) : Q(y) \cap G(x) \cap Q(x)|.$$
Plainly $Q(y) \cap G(x) \cap Q(x) = Q(y) \cap Q(x) = N$ and therefore  $|N|=2^4$. The second assertion of (2) is immediate.

We  have that $Q(x) \cong 2_+^{1+6}$ and by (2) $|N|=2^4$,  so for (3) we just need to see that $N$ is elementary abelian. Observe that 
 $$\Phi(Q(x))\Phi(Q(y)) = \la e_x, e_y \ra \leqslant Q(x) \cap Q(y),$$ and therefore 
$\Phi(Q(x) \cap Q(y)) \leqslant \Phi(Q(x)) \cap \Phi(Q(y)) = 1$
and so the result follows from (1).
\end{proof}

The next three lemmas expose detailed structure of $G\{T\}$. Surprisingly the outcome of these results is
not
the identification of $G\{T\}$ but rather of $G(x)$, which we complete in Lemma~\ref{lem:qx/ex is
semisimple}
 and Proposition~\ref{prn:S}.

\begin{lemma}
\label{lemma:omnibus gt}
The following hold.
\begin{itemize}
\item[ (1) ] $F =\mathrm O_2(G(T))=(Q(x) \cap G(T))E_y$.
\item[ (2) ] $\mathrm C_{G\{T\}} (F) = E_T$ is an irreducible $G\{T\}$-module.
\item[ (3) ] $E_T=\Phi( F ) = [F,F]=\mathrm Z(F)$, in particular, $F$ is a special 2-group.
\item[ (4) ] $G\{T\}/ \mathrm F \cong S_3 \times S_3$.
\item[ (5) ] $\mathrm C_{G\{T\}} (\mathrm F  / E_T ) = \mathrm F$.
\item[ (6) ] $\mathrm C_{G\{T\}} ( N/E_T) / F \cong S_3$  and  $\mathrm C_{G\{T\}} ( N/E_T)
\cap G(T) = F$. In particular, $N/E_T$ is irreducible as a $G(T)$-module.
\item[ (7) ] Write $\overline F  = F/ E_T$. Then $\overline F = \overline {E_x} \oplus  \overline {E_y}
\oplus \overline N$ as a $G(T)$-module.
\item[ (8) ] $\mathrm C_{G\{T\}} (N) = \mathrm C_F (N)$.
\item [ (9) ] $\mathrm C_{G\{T\}} ( F/N) = F$.
\end{itemize}
\end{lemma}
\begin{proof}
(1) We have that $G\{T\} / G(T)$ is isomorphic to $S_3$, so $\mathrm O_2(G\{T \}) \leqslant G(T)$
and  the first equality of (1) holds since $G(T)$ is normal in $G\{T\}$. For the second equality, we just calculate the orders of the subgroups involved. First note that since $G(x) \cap G\{T\} = G(T)Q(x)$ we have that 
$$G(T)/ (G(T) \cap Q(x) )\cong S_4$$
 and so $|F| = 2^2  | Q(x) \cap G(T)|$. By Lemma \ref{lem:N} (1) we have that $|Q(x) \cap G(T)|=2^6$, hence $|F|=2^8$. Now since $E_y \leqslant G(T)$ we have 
 $$E_y \cap (Q(x) \cap G(T)) = E_y \cap Q(x)$$
  and so $|E_y (Q(x) \cap G(T))| = 2^8$ by Lemma~\ref{lem:ET} (2). This completes the proof of (1).

(2) Since $E_T = \langle e_x, e_y \ra$ we see that $G\{T\}$ acts irreducibly on $E_T$.
 Let $C= \mathrm C_{ G\{T\} } (F)$. Then $C$ centralises $E_T$, so $C \leqslant G(T)$. Since 
$$E_x \leqslant Q(x) \cap G(T) \leqslant F$$
 we have $C \leqslant \mathrm C_{G(x)} (E_x) = E_x$, and in particular, $C=\mathrm Z(F)$. Similarly, $C \leqslant E_y$, so $C \leqslant E_x \cap E_y =  E_T$. Since $C$ is non-trivial and $E_T$ is an irreducible $G\{T\}$-module, we see that
$C=E_T=\mathrm Z(F)$.
 
(3) If $P$ is a $p$-group with normal subgroups $A$ and $B$ such that $P=AB$, then it can be shown that $\Phi(P)=\Phi(A)\Phi(B)[A,B]$. We use this identity and (1) to see that
$$\Phi(F) = \Phi(E_y)\Phi(Q(x) \cap G(T)) [ E_y, Q(x) \cap G(T)] \leqslant \la e_x \ra E_T = E_T.$$
This implies $[F,F] \leqslant \Phi(F) \leqslant E_T$. Since $\Phi(F)$ and $[F,F]$ are normal in $G\{T\}$,
and $E_T$ is an irreducible $G\{T\}$-module (3) holds unless $F$ is abelian.  However if $F$ is abelian then, using
$E_xE_y\leqslant F$, we see that $E_y \leqslant \mathrm C_{G(x)}(E_x) = E_x$, a contradiction. 

(4) In the proof of (1) we saw that $G(T)/F \cong S_3$. Working in $G\{T\} / F$ we see that the
centraliser of $G(T)/F$ meets $G(T)/F$ trivially, and therefore complements $G(T)/F$ in $G\{T\}/F$. Since
$(G\{T\}/F) / (G(T)/F) \cong S_3$ we obtain (4).

(5) By (3) we have that $[F,F] = E_T$, so $F \leqslant \mathrm C_{G\{T\}} (F/E_T)$. Hence if (5) doesn't hold,
then by (4) there is $D \leqslant G\{T\}$ of  order three that centralises $F/E_T$. Now (3) says that $D$
centralises $F/ \Phi(F)$ and by \cite[5.3.5, pg.180]{gorenstein} $D$ must centralise $F$ which is a contradiction to (2). Thus (5) holds.

(6) By (3) we have $[F,N] \leqslant [F,F] = E_T$ so $F \leqslant \mathrm C_{G\{T\}} ( N/ E_T)$. Since $G(T)/F$ is isomorphic to $S_3$, if $F < \mathrm C_{G(T)}(N/E_T)$ then a Sylow 3-subgroup $X$ of $G(T)$  centralises $N/E_T$.  In particular, $N/E_T \cong C_2^2$ is contained in $\mathrm C_{Q(x)/E_T}(X)$. As an $X$-module $Q(x)/E_T$ is semisimple with two 2-dimensional summands and one trivial summand, hence $|\mathrm C_{Q(x)/E_T}(X) | = 2$, a contradiction. Since $\mathrm {Aut}(N/E_T) \cong S_3$ we obtain (6) after using (4).

(7) Notice that $E_yE_x \cap N \leqslant E_y E_x \cap Q(x) = E_x(E_y \cap Q(x)) = E_x$. Similarly, we obtain
$E_y E_x \cap N \leqslant E_y$, and so $E_y E_x \cap N \leqslant E_x \cap E_y = E_T$. Thus $\overline{E_x
E_y} \cap \overline N = 1$. Now $\overline {E_x E_y} = \overline{E_x } \oplus \overline {E_y}$. All of these
spaces are $G(T)$ invariant, which yields (7).

(8) Using  $\mathrm C_{G\{T\}} (E_T) \leqslant G(T)$, part (6) and the fact that
$$\mathrm C_{G\{T\}}(N) \leqslant \mathrm C_{G\{T\}} (E_T) \cap \mathrm C_{G\{T\}}(N/E_T)$$
 we obtain $N \leqslant \mathrm C_{G\{T\}}(N) \leqslant F$. This is (8).

(9)  By part (6) we can choose $d\in G\{T \}$ of order three with $d\notin G(T)$ so that
$[N,d] \leqslant E_T$. Suppose that $d\in \mathrm C_{G\{T\}} ( F/N)$. Then $[F,d,d] \leqslant [N,d]
\leqslant E_T \leqslant E_x$. Since $E_x \leqslant F$ this implies that $d$ normalises $E_x$, a
contradiction to Lemma~\ref{lem:lem1} (since $\la G(x), d \ra = \la G(x), G\{T\} \ra=G$). Hence we have that $\mathrm C_{G\{T\}}
(F/N) \leqslant G(T)$. If there is   $d\in G(T)$ of order three so that $[F,d] \leqslant N$ then $[E_x,d]
\leqslant E_x \cap N = E_T$. Since $G(T)$ centralises $E_T$ we see that $[E_x,d,d]= [E_x,d] =1$, a
contradiction to $\mathrm C_{G(x)}(E_x) = E_x$. Hence $\mathrm C_{G\{T\} } (F/N) \leqslant F$ and since $F/N$ is abelian we see that equality holds.
\end{proof}

\begin{lemma}
\label{lem:selfcent}
The following hold:
\begin{itemize}
\item[(1)] $F/N$ is an irreducible $G\{T\}/ F$-module;
\item[(2)] $N$ is self-centralising in $G\{T\}$.
\end{itemize}

\end{lemma}
\begin{proof}
Set $\overline F = F/N$. By Lemma~\ref{lemma:omnibus gt} we see that $\overline F  = \overline {E_x} \oplus
\overline {E_y}$  as a $G(T)$-module. Moreover there is a third submodule $\overline {E_z}$ which (by the
2-transitive action of $G\{T\}$ on $\{x,y,z\}$) is distinct from $\overline{E_x}$ and $\overline{E_y}$. Part
(4) of Lemma~\ref{lemma:omnibus gt} shows that these three submodules are permuted transitively by $G\{T\}$,
so we conclude that $\overline F$ is an irreducible $G\{T\}$-module and part (9) of the same lemma shows
that  $\overline F$ is an irreducible $G\{T\}/F$-module.

By (1) we have $\mathrm C_F(N) = N$ or $F$, but the latter is ruled out by  part (3) of
Lemma~\ref{lemma:omnibus gt}. Now Lemma~\ref{lemma:omnibus gt} (8)  shows that $\mathrm
C_F(N) = \mathrm C_{G\{T\}}(N)$ which is (2).
\end{proof}

\begin{lemma}
\label{lem:g(t) splits}
Let $X$ be a Sylow 3-subgroup of $G\{T\}$. Then $\mathrm N_{G\{T\}}(X) \cong S_3 \times S_3$. In particular,
both $G\{T\}$ 
and $G(T)$ split over $F$.
\end{lemma}
\begin{proof}
Let $X$ be a Sylow 3-subgroup of $G\{T\}$ and let $\overline{G\{T\}} = G\{T\}/F$. Since $(|F|,|X|)=1$ we
have 
$$\mathrm N_{\overline{G\{T\}}}(\overline{X}) = \overline{ \mathrm N_{G\{T\}}(X)}.$$
Now Lemma~\ref{lemma:omnibus gt} (4) shows that $\overline {X}$ is normal in $\overline{G\{T\}}$ whence
$G\{T\} = \mathrm N_{G\{T\}}(X) F$. Note that $\mathrm N_{G\{T\}}(X) \cap  F = \mathrm C_{F}(X)$.
We claim that $\mathrm C_{F}(X)=1$, from this it follows that $\mathrm N_{G\{T\}}(X) \cong  S_3
\times S_3$ and we obtain the lemma.

Indeed, since the action of $X$ on $F$ is coprime we have $\mathrm C_{F/N}(X) = \mathrm C_{F}(X) N / N$ 
and Lemma~\ref{lem:selfcent}(1) shows that $\mathrm C_{F/N}(X)=1$, thus $\mathrm C_{F}(X) = \mathrm C_{N}(X)$.  By Lemma~\ref{lemma:omnibus gt} (4) and (6) we have $\mathrm C_{N/E_T}(X) =1$ so that $\mathrm C_N(X) = \mathrm C_{E_T}(X)$. Finally we see that $X$ is transitive on $\{x,y,z\}$ and therefore on $\{e_x,e_y,e_z\}$, hence  $\mathrm C_{E_T}(X)  = 1$.
This proves our claim.
\end{proof}

By \cite[Lemma 3.1(iii)]{ivanshpect} there is a unique indecomposable $\mathrm{GF}(2)L_3(2)$-module which is
an extension of $V$ by $V^*$, for $V$ the natural module. We give some brief details of a construction of
this module. Let $\epsilon : V^* \times V^* \rightarrow V$ be defined as follows for $\alpha, \beta \in V^*$
:
$$\epsilon : (\alpha, \beta) \mapsto \left \{ 
\begin{array}{c c } 0 & \text{ if } \alpha =\beta \text{ or if one of } \alpha, \beta \text{ is } 0, \\
 (\ker \alpha \cap \ker \beta)^{\#} & \text{ otherwise.}  \end{array}
 \right .$$
 Now we define  
 \begin{eqnarray}
W & = & \{ (v, \alpha ) \mid v \in V, \alpha \in V^* \} \label{w defn}
\end{eqnarray} 
  and for $v,w \in V$ and $\alpha , \beta \in V^*$ we set
$$(v,\alpha) + (w,\beta) = (v +w + \epsilon(\alpha,\beta), \alpha + \beta ).$$
It is easy to check that $W$ is an elementary abelian $2$-group and the actions of $L_3(2)$ on $V$ and $V^*$
induce an action on $W$. Moreover,  
$$V_0:=\{ (v,0) \mid v \in V\} \cong V$$
is the only nontrivial proper submodule of $W$ whilst $W/V_0 \cong V^*$. Finally $L_3(2)$ respects a unique
quadratic form $q_W$ defined on $W$ by
\begin{eqnarray}
q_W(v, \alpha ) = \left \{ \begin{array}{c c } 0  & \text{ if } \alpha =0 \text{ or if } \alpha \neq 0
\text{ and } v \notin \ker \alpha, \\
1 & \text{ if } \alpha \neq 0 \text{ and } v \in \ker \alpha . \end{array} \right .  \label{qw defn}
\end{eqnarray}

\begin{lemma}
\label{lem:qx/ex is semisimple}
As a module for $G(x)/Q(x) \cong L_3(2)$ we have $$Q(x)/\la e_x \ra \cong V \oplus V^*.$$
\end{lemma}
\begin{proof}
We set $\overline{Q(x)} = Q(x)/\la e_x \ra $ and use the bar notation. By Lemma~\ref{lem:dualaction} and the uniqueness of $W$ proved in \cite[Lemma 3.1(iii)]{ivanshpect} the only possibility other than $\overline{ Q(x) } \cong V \oplus V^*$ is that $\overline{Q(x)}$ is the module $W$ defined in (\ref{w
defn}) and $\overline{E_x}$ is the unique submodule of $\overline{Q(x)}$ of dimension three. Moreover the quadratic
form $q_E$ defined on $\overline{Q(x)}$ that $G(x)/Q(x)$ respects  is given by 
\begin{eqnarray}q_E: & \overline{a} \mapsto a^2, &  a \in Q(x). \label{qE defn}
\end{eqnarray}
Therefore, since $E_x$ is elementary abelian, $q_E( \overline{E_x } )=0$. Let $\phi$ be the $G(x)/Q(x)$-isomorphism $\phi : \overline{Q(x)} \rightarrow W$. Then by the uniqueness of $q_W$, for all $a \in Q(x)$ we have
$$q_E(\overline{a}) = q_W(\phi( \overline{a}))$$
with $q_W$ as in (\ref{qw defn}).

 Let $S\cong S_3$ be a subgroup of $G(T)$ provided by Lemma~\ref{lem:g(t) splits}. Since $S \cap Q(x) \leqslant
S \cap \mathrm O_2(G(T)) = 1$ we see that $S Q(x) /Q(x) \cong S_3$ and so $S$ acts on $\overline{E_x}$ as a
stabiliser some non-zero vector $v \in \overline{E_x}$ and some 2-space $U\leqslant \overline{E_x}$ such
that $v \notin U$. Since $S\leqslant G(T)$ we have that $S$ normalises $N $ which is
elementary abelian and has order $2^4$ by Lemma~\ref{lem:N}. Moreover $S$ centralises $E_T = \la e_x,
e_y, e_z \ra$ which is contained in $N$. By Lemma~\ref{lem:selfcent} we have that $[N,S] \neq 1$ and so we conclude $N=E_T \times [N,S]$.

%Let $X$ be a Sylow 3-subgroup of $S$. We claim that $[N,X] \neq 1$. Assume this is false. Then we have
%$[NE_x/E_x,X]=1$. By Lemma~\ref{lem:ET}(2) 
%$$E_T  \leqslant N \cap E_x \leqslant Q(y) \cap E_x = E_T,$$
%hence $NE_x / E_x$ has order $2^2$. Since $|Q(x)/E_x|=2^3$ we see that $X$ centralises $Q(x)/E_x$.
% Since $X$ is a Sylow 3-subgroup of $G(x)$ this
%contradicts Lemma~\ref{lem:dualaction}. Hence $[N,X]$ has order $2^2$. Moreover, since $N = \la e_x, e_y \ra
%\times [N,X]$ we obtain $[N,X]=[N,S]$.

Set $M:=[N,S]$ and observe that  $M=[N,S] \leqslant [Q(x),S]$. Note that $\overline{M} \cong M$ since $M
\cap \la e_x \ra = 1$. Hence $\overline{M}$  is a 2-dimensional subspace of $\overline{Q(x)}$ which we claim
satisfies the
following three properties:
\begin{itemize}
\item[(1)] $\overline{M}$ is a faithful $S$-submodule of $\overline{Q(x)}$;
\item[(2)] $\overline{M}$ is totally isotropic, that is, $q_E(\overline{M})=0$;
\item[(3)] $\overline{M}$ is not contained in $[ \overline{ E_x}  , S]$.
\end{itemize}
We have (1) by definition of $M$ and observe that (2) holds since $M$ is elementary abelian. If (3) is false
then
 $\overline{M} \leqslant \overline{E_x}$ and this implies $N = M \la e_x, e_y \ra \leqslant E_x$ which gives
$E_x = N$, a contradiction to
Lemma~\ref{lem:lem1}. Hence indeed $\overline{M}$ satisfies (1), (2) and (3).

Now $J:=[\overline{Q(x)} , S]=\overline M \oplus [\overline {E_x},S]$ is 4-dimensional and every faithful
$S$-submodule of $\overline{Q(x)}$ must be contained in $J$. Moreover, as $S$-modules, $[\overline{E_x},S]
\cong
\overline{M}$, so there are exactly three $S$-submodules of $J$.
Since $S$ preserves the decomposition $\la v \ra \oplus U$ of $\overline{E_x}$ we have that
$[\overline{E_x},S] = U$. Write $U=\{u_1,u_2,u_3,0\}$  and then let $U_1=\phi(U)$. We label the elements of $U_1$ as follows (recall (\ref{w defn}))
$$U_1 = \{ (0,0), (u_1,0), (u_2,0),(u_3,0) \}.$$

For each $u_i$ we have a 2-space $V_i := \la v, u_i \ra \leqslant \overline{E_x}$ and $S$  permutes the
subspaces 
$V_1$, $V_2$, $V_3$ transitively. Define $\alpha_i : V \rightarrow V/V_i$ (that is, $\alpha_i \in V^*$)
 and observe that $S$ has equivalent actions on
$\{u_1,u_2,u_3\}$ and $\{\alpha_1,\alpha_2,\alpha_3\}$. So we may define 
$$U_2 := \{ (0,0), (v,\alpha_1),(v,\alpha_2),(v,\alpha_3) \}$$
and note that $U_2$ is an $S$-invariant subset of $W$. Since $\alpha_i + \alpha_j = \alpha_k$ and
$\epsilon(\alpha_i,\alpha_j)= v$ for $\{i,j,k\}=\{1,2,3\}$, a quick calculation shows that $U_2$ is in fact
a subspace of $W$. Finally, we set
$$U_3 := \{ (0,0), (u_1 + v, \alpha_1), (u_2 + v, \alpha_2), (u_3+v,\alpha_3) \}$$
and again observe that $U_3$ is an $S$-invariant submodule of $W$. Since $S$ acts non-trivially on $U_2$ and
$U_3$, we have $U_2=[U_2,S]$ and $U_3 = [U_3, S]$. Thus $U_2, U_3 \leqslant [W,S]$ and $U_1$, $U_2$ and $U_3$ are the three non-trivial $S$-submodules of $[W,S]$.

Now $[\overline{E_x},S]$ and $\overline{M}$ are both totally isotropic subspaces of $J$. On the other hand,
we have $q_W(U_2) \neq 0$ since, by (\ref{qw defn}), $q_W(v, \alpha_1)=1$ and $q_W(U_3) \neq 0$ since $q_W(u_1 + v, \alpha_1) = 1$. Since $[\overline{Q(x)},S]$ and $[W,S]$ are isomorphic as $S$-modules, we see that $[\overline{Q(x)},S]$ has a unique totally isotropic subspace that is $S$-invariant. Then
properties (1) and (2) imply that $[\overline{E_x},S] = \overline{M}$, a contradiction to (3) which
completes the proof.
\end{proof}

We will denote by $E^x$  the unique normal subgroup of $G(x)$ of order $2^4$ which is contained in $Q(x)$
and is not equal to $E_x$, thus we have
$$Q(x)/\la e_x \ra = E^x/\la e_x \ra \oplus E_x / \la e_x \ra $$
as a $G(x)/Q(x)$-module.

For the next result we need information about the 1-cohomology of $L:=L_3(2)$ on the natural
module $V$ (we refer the reader to \cite[\S17 ]{aschbachersbook} for any unexplained notation). By \cite[Lemma 3.1(i)]{triextraspecial} we have $\mathrm H^1(L,V) \cong C_2$.
That is, there  exists a  unique indecomposable $L$-module $W$ which has a submodule isomorphic to $V$ and
such that $[W/V,L]=1$ (see \cite[(17.11)]{aschbachersbook}). Choosing $M=\mathrm N_{L}(R)$ for some Sylow $7$-subgroup $R$ of $L$ we observe that
$W$
is a semisimple $M$-module. Picking $w \in W$ such that $W = \la w \ra \oplus V$ (as an
$M$-module) we define the 1-cocycle $\mu : L \rightarrow V$ by 
$$\mu : \ell \mapsto [w,\ell] \in V.$$ 
Since $M$ is a maximal subgroup of $L$ we have $M=\mathrm C_{L}(w)$, in particular, $\mu(\ell) \neq 0$ for
$\ell\notin M$. This gives us a concrete description of $\mathrm H^1(L,V) = \la \mu \ra$. Similarly we define
$\gamma : L \rightarrow V^*$ so that $\gamma(m)=0$ for $m\in M$ and $\gamma(\ell) \neq 0 $ for $\ell \notin M$.
Thus we have
$$\mathrm H^1(L,V \oplus V^*) = \la \mu, \gamma \ra$$
(although we have identified $\mu$ and $\gamma$ with their images in the 1-cohomology group).
As in \cite[(17.1)]{aschbachersbook}  for a 1-cocycle $\alpha  : L \rightarrow V \oplus V^*$ we set 
$$S(\alpha) = \{ (\alpha(\ell), \ell) \mid \ell \in L\} \leqslant (V \oplus V^*) \rtimes L.$$

Then $S(0)$, $S(\mu)$, $S(\gamma)$ and $S(\mu + \gamma)$ are the \emph{standard complements} and form a
% complete set of representatives (without repetition) for
transversal of the conjugacy classes of complements to $V \oplus
V^*$ in the semidirect product. Note that the intersection of each pair of these groups is precisely $M$.

\begin{proposition}
\label{prn:S}
In $G(x)$ there exist three conjugacy classes of subgroups isomorphic to $L_3(2)$ and one
conjugacy class of subgroups isomorphic to $SL_2(7)$. Moreover, there is a unique class of complements to $Q(x)$
in  $G(x)$ for which both $E_x$ and $E^x$ are semisimple. In particular, $G(x)$ is isomorphic to the centraliser of a $2A$-involution in $L_5(2)$.
\end{proposition}
\begin{proof}
We set $\overline{G(x)}=G(x) / \la e_x \ra$ and use the bar notation. Since 
$$\mathrm C_{G(x)}(Q(x)) = \la e_x \ra$$
 we identify $\overline { G(x) }$ with a subgroup of $\mathrm{Aut}(Q(x)) \cong 2^6. O^+_6(2)$ which contains $\overline{Q(x)}$, the normal subgroup of order $2^6$. Since $G(x)$ is perfect, $\overline{ G(x) }$ is contained in the derived subgroup of $\mathrm{Aut} (Q(x))$ which is isomorphic to $2^6 \rtimes  A_8$. In particular, we see that $\overline{ G(x) }$ splits over $\overline { Q(x) }$,
  which means that we may identify $\overline{G(x)}$ with $(V \oplus
V^*) \rtimes L_3(2)$. Clearly if $L \leqslant G(x)$ and $L \cong L_3(2)$ then $\overline L
\cong L_3(2)$, whereas if $L \cong {SL}_2(7)$ then we have $L \cap Q(x) = \mathrm Z(L) =
\la e_x \ra$ so that $\overline L \cong L_3(2)$ also.

From the remarks above we know that there are four conjugacy classes of subgroups of $\overline{G(x)}$
isomorphic to $L_3(2)$. Moreover,  our standard complements are  representatives of each class,
let them be $\overline{ L_1 }$, $\overline{ L_2 }$, $\overline{ L_3 }$ and $\overline{ L_4 }$. There is a
subgroup $\overline M$ of $\overline {G(x)}$ with $\overline M \cong C_7 \rtimes C_3$ such that
$\overline{ L_i  }\cap\overline{  L_j }=\overline M$ for  $i \neq j$. Since the $ \overline { L_i }$ belong
to distinct conjugacy classes of $\overline { G(x) }$ we see that the preimages $L_i$ of the $\overline{ L_i
}$ are also non-conjugate. Moreover, a preimage is isomorphic to one of $ {SL}_2(7)$ or $C_2
\times L_3(2)$. We will show that exactly one of the $  L_i $ is  isomorphic to
${SL}_2(7)$.

Let $S=\la X,e \ra$ be the subgroup of $G(T)$ isomorphic to $S_3$ with Sylow 3-subgroup $X$ delivered by
Lemma~\ref{lem:g(t) splits}. Then $\overline{S} \leqslant \mathrm
N_{\overline{G(x)}}(\overline{X})$. Moreover, since $(|\overline{X}|, |\overline{Q(x)}|)=1$ and the
normaliser of a Sylow 3-subgroup of $L_3(2)$ is isomorphic to $S_3$, we have that 
$$\mathrm N_{\overline{G(x)}}(\overline X) = \mathrm{N}_{\overline{Q(x)}}(\overline X) \overline{S} =
\mathrm C_{\overline{Q(x)}}( \overline X) \overline{S}.$$
It follows from Lemma~\ref{lem:qx/ex is semisimple} that $$\mathrm
C_{\overline{Q(x)}}(\overline
X)=\mathrm C_{\overline{Q(x)}}( \overline S)= \la \overline{e_y}, \overline{r} \ra $$
where $\overline r$ is some element of $\overline {E^x }$ which, when we identify $\overline { E_x }$ and
$\overline { E^x }$ with $V$ and $V^*$ respectively, is a line on which $\overline { e_y }$ does not lie. In
particular, $\overline r$ is contained in a totally isotropic subspace and $[e_y, r ] \neq 1$.
Now $\mathrm N_{\overline{G(x)}} (\overline{ X })$ contains exactly four subgroups isomorphic to $S_3$, namely
\begin{eqnarray*}
\overline{ N_1 } & = &  \la \overline X, \overline {e } \ra , \\
\overline{ N_2 } & = &   \la \overline X, \overline{ e  e_y} \ra, \\
\overline{ N_3 } & = &   \la \overline X , \overline {e  r } \ra, \\
\overline{ N_4}  & = &   \la \overline X , \overline{ e  e_y r } \ra.
\end{eqnarray*}

By Sylow's Theorem, Sylow 3-subgroups of $\overline M$ and $\overline X$ are conjugate, so we may assume
that $\overline X \leqslant \overline M$. Now we see that $\mathrm N_{\overline {L_i} }  (\overline X) =
\overline { N_j }$ for some $j$. Since $\overline {L_i} \cap \overline {L_k} = \overline M$ for $i \neq k$,
we may relabel the $\overline {N_j}$ so that $\mathrm N_{\overline {L_i}}(\overline X) = \overline {N_i}$
for $i=1,2,3,4$. 
Since $e\in G(T)$ we have that $e$ centralises $e_y$, moreover, since $e$ is an involution, both $e$ and
$e_y e$ are involutions,
so the preimages of $\overline{N_1}$ and $\overline{N_2}$ are isomorphic to $C_2 \times S_3$. Now $e$ normalises $\langle r, e_x \rangle$, 
so either $r^e = r$ or $r^e = e_x r$. In the latter case we see that $r^{ee_y}=r$. Without loss of generality therefore, we can assume that $r^e = r$. Hence the preimage of $N_3$ is isomorphic to $C_2 \times S_3$ also. Now $(ee_yr)^2 = e_x$ so the preimage of $N_4$ is isomorphic to $C_3 \rtimes C_4$.
Considering the number of involutions in the $L_i$, we see that  $L_1$, $L_2$ and $L_3$ are isomorphic
to $C_2 \times L_3(2)$ and $L_4 \cong  {SL}_2(7)$.

We  choose a preimage $M$ of $\overline M$ so that $X \leqslant M$ and $M \leqslant L_i \cap L_j$ for $i,j \in \{1,2,3,4\}$. As $M$-modules we have
$$E_x = \la e_x \ra \oplus [E_x,M] \text{ and } E^x = \la e_x \ra \oplus [E^x,M].$$
 As $N_i$-modules, we have
$$E_x = \la e_x, e_y \ra \oplus [E_x,X] \text{ and } E^x = \la e_x,r \ra \oplus [E^x,X]$$
with $[E_x,X] \cong [E^x,X] \cong C_2^2$ and of course $[E_x,X] \leqslant [E_x,M] $ and $[E^x,X] \leqslant [E^x,M]$. Since $L_i = \la M, z_i \ra $, where $z_1 = e$, $z_2= ee_y$, $z_3 = er$ and $z_4 = ere_y$,  to determine $[E_x,L_i]$ we just need to evaluate $[z_i,e_y]$ and $[z_i, r]$ for each $i$.
Since $[e_y,r] = e_x$ we have that 
\begin{itemize}
\item[] $[E_x,L_1] = [E_x, L_2] =[E_x,M]$,
\item[] $[E^x,L_1] = [E^x,L_3] = [E^x,M]$,
\item[] $[E_x, L_3] = [E_x,L_4] = E_x$  and
\item[] $[E^x,L_2] = [ E^x, L_4] = E^x$.
\end{itemize}
 Hence both $E_x$ and $E^x$ are semisimple modules for $L_1$ only. Clearly the decompositions are invariant under conjugation by $G(x)$, so we obtain the second part of the proposition.

We have now seen that $G(x)$ is a split extension of the extraspecial group $Q(x) \cong 2_{+}^{1+6}$ by $L_3(2)$ and have completely determined the action of a complement on $Q(x)$. This uniquely determines the isomorphism type of $G(x)$. After inspecting the centraliser of a $2A$-involution in $L_5(2)$, we see that this group is isomorphic to $G(x)$.
\end{proof}

We now introduce some notation for subgroups of $G(x)$.

\begin{notn}
\label{notnforgx}
Recall that we can identify $G(x)$ with the centraliser of an involution in
$L_5(2)$. Hence $G(x)=\la a_1,a_2,\ldots,a_{12}\ra$ where $a_i$ is the
matrix with 1's on the diagonal and 0's everywhere else except for the
coordinate given by $a_i$ in the matrix below.
$$\left(\begin{array}{ccccc}
         1  & 0   & 0      & 0      & 0 \\
         a_1& 1   & a_{11} & 0      & 0\\
         a_2& a_3 & 1      & a_{12} &0 \\
         a_4& a_5 & a_6    & 1      &0 \\
         a_7& a_8 &a_9     & a_{10} & 1
         \end{array}\right)$$
%The commutator relations are those arising naturally from multiplying
%the corresponding matrices and yield the bilinear form $f$ to which $q_E$ polarises (recall (\ref{qE defn})). We also have the following relations which we will make use of in Section~\ref{sec:realisations}:
%$$(a_3a_{11})^3=(a_6a_{12})^3=(a_{11}a_{12})^4=1.$$
With this identification we have $e_x=a_7$, $e_y=a_4$,   $E_T = \langle e_x, e_y \rangle$, 
$E_x=\la a_1,a_2,a_4,a_7\ra$, $E^x = \la a_7, a_8, a_9, a_{10} \ra$,
$Q(x)=\la a_1,a_2,a_4,a_7,a_8,a_9,a_{10}\ra$ and 
$L=\la a_{11},a_{12},a_3,a_5,a_6\ra$.   Now $G(x)\cap G\{T\}$ is the stabiliser in $G(x)$ of $E_T$ and in particular $(G(x)\cap
G\{T\})/Q(x)\cong S_4$ is the stabiliser in $G(x)/Q(x)\cong L_3(2)$ of $E_T/\langle a_7\rangle$. Since
$G(x)$ acts dually on $E_x/\langle e_x\rangle$ and $E^x/\langle e_x\rangle$, it follows that $G(x)\cap
G\{T\}$ fixes a 3-subspace $U$ of $E^x$ containing $\langle a_7\rangle$.  Moreover,  we must have that
$U=\langle a_7,a_8,a_9\rangle$. We note that $E_T/ \langle  a_7\rangle$ and $U/\langle e_x\rangle$ are
respectively the unique 1-dimensional and 2-dimensional subspaces of $Q(x)/\langle a_7 \rangle$ fixed by
$G(x)\cap G\{T\}$. There are exactly four elementary abelian normal subgroups of $G(x)\cap G\{T\}$ of order $2^4$, these are
$E_x,E^x,W_1=\langle a_7,a_4,a_8,a_9\rangle$ and $W_2=\langle a_7, a_4, a_1 a_9, a_2 a_8\rangle$. (The number of such subgroups can either be directly calculated using a computer algebra package, or see Lemma~\ref{lem:tau outside caqx inn}.)

Now $N$ is  an elementary abelian subgroup of order $2^4$ normalised by $G(x)\cap
G\{T\}$ with $N\cap E_x=E_T$. Then $N/\langle a_7\rangle$ is totally singular and so contained in $\langle
a_4\rangle^\perp=\langle a_4, a_2, a_1, a_8,a_9\rangle$. Note that $W_2$ is
self-centralising in $G(x)\cap G\{T\}$ while $W_1$ is not, thus Lemma~\ref{lem:selfcent}(2) implies $N=W_2$.

%Let $M_1, M_2$ be the dual modules of $L$ in $Q(x)/\la e_x\ra$, and
%$P_1$, $P_2$ be their respective preimages in $Q(x)$. Then 
%$P_1\cong P_2\cong C_2^4$, and we can take 
%$P_1=E_x=\la a_1,a_2,a_4,a_7\ra$ and $P_2=\la a_7,a_8,a_9,a_{10}\ra$.
%Note that $P_1$ and $P_2$ are semisimple $L$-modules, that is, $L$
%decomposes both $P_1$ and $P_2$ into $2\oplus 2^3$. By \cite{??}, there
%are two other classes of complements for $Q(x)$ in $G(x)$. For a
%subgroup in one class, $P_1$ is semisimple while $P_2$ is
%indecomposable, while for a subgroup in the other class $P_1$ is
%indecomposable and $P_2$ is semisimple.
\end{notn}
%
%\begin{notn}
%\label{notnforgx}
%We write $Q(x)=\langle e_x,e_y,e_u,e_v,f_y,f_u,f_v\rangle$ such that $E_x=\langle e_x,e_y,e_u,e_v\rangle$,
%$E^x=\langle e_x,f_y,f_u,f_v\rangle$,  and $f(e_w,f_z)=0$ when $w=z$ and 1 otherwise where $f$ is the
%bilinear form given by the extraspecial group $Q(x)$  to which $q_E$ polarises (recall (\ref{qE defn})). Then $E_T=\langle
%e_x,e_y\rangle$.  Now $G(x)\cap G\{T\}$ is the stabiliser in $G(x)$ of $E_T$ and in particular $(G(x)\cap
%G\{T\})/Q(x)\cong S_4$ is the stabiliser in $G(x)/Q(x)\cong L_3(2)$ of $E_T/\langle e_x\rangle$. Since
%$G(x)$ acts dually on $E_x/\langle e_x\rangle$ and $E^x/\langle e_x\rangle$, it follows that $G(x)\cap
%G\{T\}$ fixes a 3-subspace $U$ of $E^x$ containing $\langle e_x\rangle$.  Moreover,  we must have that
%$U=\langle e_x,f_u,f_v\rangle$. We note that $E_T/ \langle  e_x\rangle$ and $U/\langle e_x\rangle$ are
%respectively the unique 1-dimensional and 2-dimensional subspaces of $Q(x)/\langle e_x\rangle$ fixed by
%$G(x)\cap G\{T\}$. The elementary abelian normal subgroups of $G(x)\cap G\{T\}$ of order $2^4$ are then
%$E_x,E^x,W_1=\langle e_x,e_y,f_u,f_v\rangle$ and $W_2=\langle e_x,e_y,e_u+f_v,e_v+f_u\rangle$.
%
%
%Now $N$ is  an elementary abelian subgroup of order $2^4$ normalised by $G(x)\cap
%G\{T\}$ with $N\cap E_x=E_T$. Then $N/\langle e_x\rangle$ is totally singular and so contained in $\langle
%e_y\rangle^\perp=\langle e_y,e_u,e_v,f_u,f_v\rangle$. Note that $W_2$ is
%self-centralising in $G(x)\cap G\{T\}$ while $W_1$ is not, thus Lemma~\ref{lem:selfcent}(2) implies $N=W_2$.
%\end{notn}

\begin{lemma}
\label{lem:gt isom}
$G\{T\}$ splits over $N$ and $G\{T\}=N\rtimes K$, where 
$K\cong C_2^4\rtimes (S_3\times S_3)$ is the stabiliser in $L_4(2)$
  of the subgroup $E_T\cong C_2^2$ of $N$. 
\end{lemma}
\begin{proof}
We first show that $G\{T\}$ splits over $N$. Since $N$ is a normal abelian subgroup of $G\{T\}$, by 
Gasch\"{u}tz's Lemma  \cite[(10.4)]{aschbachersbook} it suffices to show that a Sylow 2-subgroup of $G\{T\}$
splits over $N$. Such a subgroup is contained in $G(x) \cap G\{T\}$. We will use Notation~\ref{notnforgx}.
Then a complement to $N$ is given by $\la a_8,a_9,a_{10} \ra \rtimes \mathrm{Dih}(8)$ where the
$\mathrm{Dih}(8)$ subgroup is a Sylow 2-subgroup of a standard complement which preserves both $E_x$ and
$E^x$ as semisimple spaces.

Let $K$ be a complement to $N$ in $G\{T\}$. Then $K \cong KN/N$ can be identified with a subgroup of a 2-space
stabiliser in
${GL}_4(2)= L_4(2)$, since $E_T$ is a normal subgroup of $G\{T\}$. By
comparing orders, we see that $K$ is the full stabiliser in $L_4(2)$ of a 2-dimensional subspace and so
$K\cong C_2^4\rtimes (S_3\times S_3)$. 
\end{proof}

Hence we have proved the following.
\begin{proposition}
\label{thm:stabs}
Parts (1)-(3) of Theorem~\ref{thm:main} hold.
\end{proposition}

\section{Determination of the amalgams}
\label{sec:determine of isom}

In Proposition~\ref{thm:stabs} we determined the isomorphism type of the amalgam $ \{G(x),G\{T\}\}$ appearing in Theorem~\ref{thm:main}.
 We now wish to determine the number of isomorphism classes of amalgams of
this type. Two amalgams $\cA=\{A_1,A_2\}$, $\cB=\{B_1,B_2\}$ are
\emph{isomorphic} if there exists a bijection
$\varphi:A_1\cup A_2\rightarrow B_1\cup B_2$ that maps $A_i$ onto
$B_i$ such that $\varphi(xy)=\varphi(x)\varphi(y)$ for all 
$x,y\in A_i$ and for $i=1,2$. Let $B=A_1\cap A_2$, $D=\Aut(B)$ and
$D_i$ be the image in $D$ of $\mathrm N_{\Aut(A_i)}(B)$. Then Goldschmidt's
Theorem \cite[(2.7)]{Goldschmidt80} states that the number of nonisomorphic amalgams of
the same type as $\cA$ is equal to the number of double cosets of
$D_1$ and $D_2$ in $D$.

We determine the number of isomorphism classes of amalgams of type $\{G(x),G\{T\}\}$ whilst considering an infinite family of amalgams. For $n\geqslant 4$ we define the amalgam $\mathcal U_n$ as follows. Let $H=\mathrm{AGL}_n(2)=R \rtimes L$ with $R\cong 2^n$ and $L \cong \mathrm L_n(2)$. Picking $r,s \in R^{\#}$ with $r \neq s$ we let $B_n = \mathrm N_H(\la r \ra)$ and $C_n=N_H(\la r,s\ra)$. Then set
\begin{equation}
\label{defn of um}
 \mathcal U_n = \{ B_n,C_n \}.
 \end{equation}
Each amalgam in the sequence $(\mathcal U_n)_{n\geq 4}$ is unfaithful since $B_n$ and $C_n$ both normalise the elementary abelian subgroup  $R$ of order $2^n$. Note that $\mathcal{U}_4$ and $\{G(x),G\{T\}\}$ have the same type but $\mathcal U_4$ is unfaithful while the amalgam appearing in Theorem~\ref{thm:main} is faithful since it is a weakly locally projective amalgam.

We now assemble results on  the automorphism groups of
$B_n$, $C_n$ and $B_n\cap C_n$. We let $Q_n=O_2(B_n)$ and observe that $Q_n\cong 2^{1+2(n-1)}_+$. In particular, $Q_n=E_1E_2$ where $s\in E_1$, $E_1\cong E_2\cong 2^n$ and $E_1\cap E_2=Z(Q_n)=\langle r\rangle$. Moreover, $B_n$ is the centraliser of an appropriate involution in $L_{n+1}(2)$. We choose a subgroup $L$ of $B_n$ such that
$$B_n = Q_n \rtimes L$$
and $L \cong L_{n-1}(2)$ decomposes both $E_1$ and $E_2$ into semisimple modules. Note that  $E_1/\la r \ra$ and $E_2/\la r \ra$ are dual and the module $Q_n/\la r \ra = E_1/\la r \ra \oplus E_2 / \la r \ra$ admits an alternating bilinear form defined  by commutators in $Q_n$.

\begin{lemma}
\label{lem:qx char}
$Q_n$ is characteristic in $B_n\cap C_n$.
\end{lemma}
\begin{proof}
Suppose for a contradiction there is $\alpha \in \Aut(B_n\cap C_n )$ such that $(Q_n)^\alpha \neq Q_n$
and let $P=(Q_n)^\alpha$. Then $Q_nP > Q_n$ and is normal in $B_n\cap C_n$. Since $B_n\cap C_n / Q_n\cong 2^{n-2}.L_{n-2}(2)$, we see that $\mathrm O_2(B_n  \cap C_n /Q_n)$ is a minimal normal subgroup of order $2^{n-2}$. This implies $|Q_n P| = 2^{n+(n-1)+(n-2)}$
and $|Q_n \cap P| = 2^{n+1}$. Since $\mathrm (Z(B_n\cap C_n))^\alpha = \mathrm Z(B_n\cap C_n) = Z(Q_n)$ we have $Z(Q_n) \leqslant
Q_n \cap P$.

Let $\overline {Q_n } = Q_n /Z(Q_n)$. Then $\overline{Q_n\cap P}$ has order $2^n$. Note that 
$$ [ Q_n \cap P, P] \leqslant [P,P] = Z(Q_n),$$ 
so $PQ_n/Q_n$ centralises $\overline{Q_n \cap P}$. Now we have
$$PQ_n/Q_n = \mathrm O_2(B_n\cap C_n/Q_n) \cong 2^{n-2}.$$ 
With $V$ the natural module for $\mathrm L_{n-1}(2)$ we see that $\overline{Q_n} =E_1/\la r \ra \oplus E_2 / \la r \ra \cong V\oplus V^*$. An easy calculation shows that subspace of fixed points of $PQ_n/Q_n$ acting on $\overline{Q_n}$  has  order $2^{n-1}$,  a contradiction to the fact that $PQ_n/Q_n$ centralises $\overline{Q_n \cap P}$ which has order $2^n$.
\end{proof}

In the next lemma we see that the amalgam $\mathcal U_4$ has properties different from $\mathcal U_n$ for $n\geqslant 5$. Recall that for $n=4$ we use Notation \ref{notnforgx} for elements of $B_4$.

\begin{lemma}
\label{lem:tau outside caqx inn}
Let $\Gamma_n=\Aut(B_n\cap C_n)$. Then 
$$|\Gamma_n: \mathrm C_{\Gamma_n}(Q_n)(\Inn(B_n\cap C_n))| = \left \{ \begin{array}{cc} 2 & \text{ if } n=4, \\ 1 & \text{ otherwise.}\end{array} \right .$$
 Moreover, if  $n=4$ and
$$\tau\in \Gamma_n -  \mathrm C_{\Gamma_n}(Q_n)(\Inn(B_n\cap C_n))$$
then $\tau$ interchanges $N$ and $E_x$.
\end{lemma}
\begin{proof}
Let $\gamma \in \Gamma_n$. By Lemma~\ref{lem:qx char},  $Q_n$ is characteristic in $B_n \cap C_n$ and so $\gamma$ normalises $Q_n$ and $\langle r \rangle=Z(Q_n)$, and therefore acts on $\overline{B_n \cap C_n} = (B_n \cap C_n)/\la r \ra$.  Since $\overline{\la r, s\ra}$ and $\overline{\langle
C_{E_2}(s)  \rangle}$ are the only totally singular 1- and $(n-2)$-spaces of $\overline{Q_n}$ fixed by $B_n \cap C_n$ each of them must be stabilised by $\gamma$. Set
$$\mathcal I = \{ \overline U \subset \overline{Q_n} \mid \overline{U} \text{ is totally isotropic, } B_n \cap C_n\text{-invariant and }\dim \overline{U} = n-1 \}.$$
Observe that $\gamma$ permutes the elements of $\mathcal I$ and $\overline{E_1}$, $\overline{E_2}$ are elements of $\mathcal I$. Let $E_3 = \la r,s, C_{E_2}(s) \ra$, then $\overline{E_3} \in \mathcal I$. Suppose that $\overline{E_4}$ is a fourth element of $\mathcal I$. Consider the quotient $\widetilde{Q_n}:=\overline{Q_n} / \overline{\la r, s \ra }$ where $\overline{E_4}$ projects to an $n-2$, respectively, $n-1$ dimensional subspace if $\overline{E_4}$ contains $\overline{\la r,s\ra}$, respectively, doesn't contain $\la r, s \ra$. We have $\widetilde{Q_n} = \widetilde{E_1} \oplus \widetilde {E_2}$ and $\widetilde{E_3}$ is the unique $(B_n \cap C_n)$-invariant proper subspace of $\widetilde{E_2}$. If $n\neq 4$ then $\widetilde{E_3}$ and $\widetilde{E_1}$ are dual non-isomorphic $(B_n \cap C_n)$-modules. Thus one of $\widetilde{E_4} = \widetilde{E_1}$, $\widetilde{E_4} = \widetilde{E_3}$ or $\widetilde{E_4}=\widetilde{E_2}$. In the respective cases we obtain $\overline{E_4}=\overline{E_1}$, $\overline{E_4}=\overline{E_3}$ or $\overline{E_4}=\overline{E_2}$, and so for $n\neq 4$ we have $|\mathcal I| = 3$. For $n=4$ we see that there is exactly one more option for $\widetilde{E_4}$, the unique diagonal submodule of $\widetilde{E_1} \oplus \widetilde{E_3}$. This is the image of $N$ defined in Notation \ref{notnforgx}, and therefore $|\mathcal{I}|=4$ for $n=4$.

Since $\la r, s\ra$ is fixed by $\gamma$ and $\overline{E_2}$ is the only element of $\mathcal I$  not containing $\overline{\la r, s \ra}$, $E_2 $ is fixed by $\gamma$. Since $E_2$ is fixed and $E_3 = \la r, s , C_{E_2}(s) \ra$ we see that $E_3$ is fixed by $\gamma$ also. Hence $E_1$ is fixed by $\gamma$ unless $n=4$ and possibly $\gamma$ interchanges $ E_1 =E_x$ and $N$ as above.

Now $\Gamma_n/\mathrm C_{\Gamma_n}(Q_n)$ is
isomorphic to a subgroup of $\Aut(Q_n)=2^{2(n-1)}.O^+_{2(n-1)}(2)$ containing $\mathrm{Inn}(Q_n)$.  The stabiliser in $O^+_{2(2n-1)}(2)$ of two complementary totally isotropic $(n-1)$-spaces is $L_{n-1}(2)$ and the stabiliser
in this of a totally singular 1-space contained in one of the $(n-1)$-spaces is $2^{n-2}.L_{n-2}(2)$. Since each element of $\Gamma_n$ fixes $\overline{E_2}$ and $\overline{\la r,s \ra }$, it follows that $\mathrm C_{\Gamma_n}(Q_n) \Inn(B_n \cap C_n) $ is the stabiliser in $\Gamma_n$ of $E_1$. In the previous paragraph we saw that $|E_1^{\Gamma_n}|=1$ if $n\neq 4$ and $|E_1^{\Gamma_n} | \leqslant 2$ if $n=4$. In particular, $\Gamma_n = \mathrm C_{\Gamma_n}(Q_n) \Inn(B_n \cap C_n) $ if $n\neq 4$ and for $n=4$ we have
 $$|\Gamma_n : \mathrm C_{\Gamma_n}(Q_n) \Inn(B_n \cap C_n) | \leqslant 2.$$

Recall Notation~\ref{notnforgx} for elements of $B_4$. We define an automorphism $\beta$ of $B_4 \cap C_4$ as follows
\begin{equation}
\beta  :   \begin{array}{ccc} a_1 &\mapsto &a_1a_7 a_9\\ a_2 & \mapsto & a_2a_7a_8 \end{array}
 \label{beta def}
\end{equation}
and $\beta$ fixes $a_3,\dots,a_{11}$. Clearly $\beta \notin \mathrm C_{\Gamma_n}(Q_n) \Inn(B_n \cap C_n )$ and thus we obtain the equality stated in the lemma.
\end{proof}

\begin{proposition}
\label{prn:GxGT}
Let $\Gamma_n=\Aut(B_n \cap C_n)$. If $n \neq 4$ then $\Gamma_n = \mathrm{Inn}(B_n \cap C_n)$. If $n = 4$ then $|\mathrm C_{\Gamma_n}(Q_n)|=2$ and  $| \Out(B_n \cap C_n) | 
= 4$.
\end{proposition}
\begin{proof}
Let $C=  \mathrm C_{\Gamma_n}(Q_n)$. Suppose there is $g\in B_n \cap C_n$ such
that $c_g \in C \cap \mathrm{Inn}(B_n \cap C_n)$ (where $c_g$ denotes the automorphism induced by
conjugation by $g$).
 Then for all $w\in Q_n$ we have $w=wc_g = w^g$, so
$g\in \mathrm C_{B_n}(Q_n) = \mathrm Z(B_n)$, whence $c_g=1$. Hence 
$$[C,\mathrm{Inn}(B_n \cap C_n) ] =1.$$
 Now suppose that $\alpha \in C$ and let $g\in B_n \cap C_n$ be arbitrary. Since $\alpha$ centralises
$c_g$ we have $c_g = c_{g\alpha}$. It follows that $g\alpha = g$ or $g\alpha = g r$ (where $\la r \ra =
\mathrm Z(B_n \cap C_n) \cong C_2$). Since $r \alpha = r$ we have that $g\alpha^2 = g$ and $\alpha^2=1$ by the
arbitrary choice of $g$. For $g_1,g_2 \notin C_{B_n \cap C_n}(\alpha)$ we have $(g_1g_2^{-1})\alpha = g_1 r g_2^{-1}r = g_1g_2^{-1}$. Hence $g_1C_{B_n \cap C_n}(\alpha)=g_2C_{B_n \cap C_n}(\alpha)$, that is, $C_{B_n \cap C_n}(\alpha)$ is a subgroup of index  at most two containing $Q_n$. If $n\neq 4$ then $B_n \cap C_n = Q_n \rtimes 2^{n-2} . L_{n-2}(2)$ has no such proper subgroup, and we have $C=1$. If $n=4$ then  $B_n \cap C_n$ has a unique such subgroup and it follows that $|C|\leqslant 2$.  We now show that  in the case of $n=4$ we have equality. Let $D$ be the unique
index two subgroup of  $B_n \cap C_n$ that contains $Q_n$. We define $\alpha : B_n \cap C_n
\rightarrow B_n \cap C_n$ by the following,
\begin{equation}\alpha : g \mapsto \left \{ 
\begin{array}{c c } g & \text{ if } g \in D \\
 g r & \text{ otherwise.}  \end{array}
 \right .
 \label{alpha def}
\end{equation}
 Since $r$ is an element of order two in the centre of $B_n$, it is easy to check that $\alpha$ is a
homomorphism. Moreover, by definition $\ker \alpha = 1$, so $\alpha$ is a non-trivial automorphism of $B_n
\cap C_n$ which lies in $\mathrm C_{\Gamma_n}(Q_n)$.

%We now have  $|C\Inn(A_n \cap B_n) : \Inn(A_n \cap B_n)|=2$ and by Lemma~\ref{lem:tau outside caqxinn}  we have $|A: C\Inn(A_n \cap B_n)|  = 2$. 
The result now follows from Lemma~\ref{lem:tau outside caqx inn}.
 \end{proof}

For the next three results we restrict our attention to $n=4$. Therefore we set
\begin{center}
 $B= B_4$,\\
  $C=C_4,$\\
   $Q=Q_4$,
   \end{center}
and we use Notation~ \ref{notnforgx} and the results of Section~\ref{sec:determine}.

\begin{proposition}
\label{lem:autG(x)}
$\Aut(B)=(Q/Z(Q)) \rtimes \Aut(L)$ where 
$(Q/Z(Q))  \cong 2^{6}$ decomposes into dual $L$-modules which are
interchanged by the inverse-transpose automorphism of $L$.
\end{proposition}
\begin{proof}
Since the centraliser of an appropriate involution in $\Aut(L_{4}(2))$
is isomorphic to $Q \rtimes \Aut(L_3(2))$ and $Z(Q) =\mathrm Z(B)$, we have that  
$Q/Z(Q) \rtimes \Aut(L_3(2))\leqslant\Aut(B )$. Let $g$ be
an automorphism of $B $. Then
$L^g$ is a complement of $Q $ and $L^g$ must decompose $Q $ in the
same manner. Since all complements of $Q $ with this property are
conjugate to $L$ by Proposition \ref{prn:S}, we can adjust $g$ by an inner automorphism so that $g$ normalises $L$ and then by an
inverse-transpose automorphism if necessary so that $g$ fixes $E_1$ and $E_2$ setwise.  However, $L$ is the
full stabiliser in
$\Aut(Q )$ of $E_1$ and $E_2$. Hence $g\in Q /Z(Q ) \rtimes \Aut(L)$.
\end{proof}

\begin{lemma}
\label{lem:Gxinduced}
The stabiliser of $B  \cap C $ in $\Aut(B )$
induces the inner automorphism group of $B  \cap C $.
\end{lemma}
\begin{proof}
Since the stabiliser of a 1-space is self-normalising in $\Aut(L)$ it follows that
$B \cap C $ is selfnormalising in $Q \rtimes \Aut(L)$. The
result follows.
\end{proof}

\begin{lemma}
\label{lem:autGT}
We have $\Aut(C )=\Inn(C )\cong C $ and the stabiliser of
  $B  \cap C $ in $\Aut(C )$ induces the inner automorphism group
of $B  \cap C $.
\end{lemma}
\begin{proof}
Write $C = N R S$ where $S$ is the normaliser of a Sylow 3-subgroup $X$ of $C$ and $RS$ is a
complement to $N$ in $C$ (so $R \cong C_2^4$). Observe that $F=O_2(C)$ and $E_T = Z(F)$ are
characteristic subgroups of $C$. We claim that $N$ is the unique normal subgroup of $C$ of order
$2^4$, and is therefore characteristic. Suppose that $M$ is another such subgroup. Then $M \leqslant F$ and
so  $M \cap Z(F)=M \cap E_T$ is  a nontrivial normal subgroup of $C$. Since $E_T$ is irreducible as a
$C$-module, we have $E_T \leqslant M$. If $M \neq N$ then $M \cap N = E_T$ by Lemma~\ref{lemma:omnibus gt} (6) and therefore $MN/N$ has
order $2^2$ and is a normal subgroup of $C/N$ contained in $F/N$. This contradicts
Lemma~\ref{lem:selfcent} (1). 

Let $\Gamma=\Aut(C)$. By the Frattini argument we have $\Gamma = \mathrm N_\Gamma(X)I$, for $I:=\mathrm {Inn}(C)$. Let
$h\in\mathrm N_\Gamma(X)$ and note that $h$ normalises $\mathrm N_I(X)=S$, which is isomorphic to $S_3 \times S_3$ by
Lemma~\ref{lem:g(t) splits}. Let  $T_1$, $T_2$ be subgroups of $S$ so that $S = T_1 \times T_2$ and $T_1
\cong T_2 \cong S_3$ and label so that $T_1 = \mathrm C_S(E_T)$ and $T_2$ acts faithfully on $E_T$. Since $E_T$ is
characteristic in $C$ we see that $h$ must normalise both $T_1$ and $T_2$. Thus, after adjusting $h$ by
an inner automorphism if necessary, we may assume $h \in \mathrm C_\Gamma(S)$.

Now $h$ normalises each of $E_T$, $N$ and $T_2$, so $h$ normalises the complement $E:=\mathrm C_N(T_2)$ to $E_T$ in
$N$. Since $E_T$ and $E$ are irreducible modules for $T_1$ and $T_2$, we see that $h$ centralises both $E_T$
and $E$, whence $h$ centralises $N$. We now claim that $h$ normalises $R$. Since $R$ is an absolutely irreducible
module for $S$ and $h$ centralises $S$, it will follow that $h$ centralises $R$ and therefore $h$
centralises $NRS = C$, from which we conclude $h \in I$ as desired. We now prove the claim.  Note that
$T_2$ centralises $N/E_T$ and that $T_2$ preserves a decomposition of $R$ into two irreducible modules.
Since $h$ acts on $F/E_T$ and normalises $[T_2,F/E_T]=RE_T/E_T$ we see that $h$ normalises $RE_T$. Now $T_1$
centralises $E_T$ and also preserves a decomposition of $R$ into two irreducible modules, thus
$[T_1,E_TR]=R$. Since $h$ normalises $E_TR$ and $T_1$ we see that $h$ normalises $R$. This proves the claim
and we obtain the lemma.
\end{proof}

We are now in a position to determine the number of isomorphism classes of amalgams of type $
\mathcal U_n$. Recall that
an amalgam $\{B,C\}$ is faithful if there is no normal subgroup contained in $B \cap C$.

\begin{theorem}
\label{thm:4}
There are four amalgams of type $\mathcal U_4$, and precisely two of these are faithful. For $n\geqslant 5$ there is a unique amalgam of type $\mathcal U_n$.
\end{theorem}
\begin{proof}
We use Goldschmidt's Theorem~\cite[(2.7)]{Goldschmidt80}. By Lemmas \ref{lem:autG(x)} and \ref{lem:autGT}
this
says that the number of isomorphism classes of amalgams of the same type as $\mathcal U_n$ is the number
of double cosets of $I:=\mathrm{Inn}(B_n \cap C_n )$ in $T:=\mathrm{Aut}(B_n \cap C_n)$. Hence for $n\geqslant 5$ there is a unique isomorphism class of amalgams of type $\mathcal U_n$ by Proposition~\ref{prn:GxGT}. Now consider the case $n=4$ and let $B=B_4$, $C=C_4$ and $Q=Q_4$. Let $\alpha$ and $\beta$ 
be the automorphisms of $B \cap C $ defined in (\ref{alpha def}) and (\ref{beta def}) respectively.
Proposition~\ref{prn:GxGT} shows that there are four cosets of $I$ in $T$.  It is easy to check that the cosets $I$, $I\alpha$, $I\beta$ and $I\alpha\beta$ are distinct.
To see then that there are exactly four amalgams of this type, we just need to show that all of these cosets
are in distinct $I$-orbits. If this is not the case, then we must have $I\alpha\beta I = I \beta I$. This
implies there are $g,h\in I$ such that $\beta = h\alpha\beta g$. That is $h^{-1} \beta g^{-1} \beta^{-1} =
\alpha$, which gives $\alpha \in I$, a contradiction.
% Finally we note that the two amalgams corresponding to
%the double cosets $I\beta I$ and $I\alpha \beta I$ are not faithful since $\beta$ must interchange $N$ and
%$E^x$ by Lemma~\ref{lem:tau outside caqx inn}.

The automorphism $\alpha$ of $B \cap C $ preserves  faithfulness since  every normal subgroup of $B$ contained in $B \cap C$ is
contained in $Q$. There exist faithful and unfaithful amalgams of type $\mathcal U_4$ inside $M_{24}$ and  $\mathrm{AGL}_4(2)$ respectively. Thus we see that exactly two of the four isomorphism classes of amalgams of type $\mathcal U_4$ are faithful.
\end{proof}

Theorem \ref{thm:main} now follows from Proposition \ref{thm:stabs} and 
Theorem \ref{thm:4}.

\section{Completions}
\label{sec:realisations}

In this final section we find presentations for the universal completions of the two faithful amalgams appearing in Theorem~\ref{thm:main} and we give finite completions for both. To derive these presentations, it is convenient to begin with an unfaithful amalgam of the same type and use Theorem~\ref{thm:4} to obtain the faithful amalgams. Recall the definition of $\mathcal U_4$ from the beginning of Section~\ref{sec:determine of isom}. 
%To this end, for $n\geqslant 3$ we define the amalgam $\mathcal U_n$ as follows. Let $G=\mathrm{AGL}_n(2)=R \rtimes C$ with $R\cong 2^n$ and $C\cong \mathrm L_n(2)$. Picking $r,s \in R^{\#}$ with $r \neq s$ we let $A_n = \mathrm N_G(\la r \ra)$ and $B_n=N_G(\la r,s\ra)$. Then set
%$$ \mathcal U_n = \{ A_n,B_n \}.$$
%Each amalgam in the sequence $(\mathcal U_n)_{n\geq 3}$ is unfaithful since $A_n$ and $B_n$ both normalise the elementary abelian subgroup  $R$ of order $2^n$. There is a unique isomorphism class of amalgams of this type for $n\neq 4$. On the other hand the amalgam $\mathcal U_4$ has the same type as the amalgams of Theorem~\ref{thm:main},  and Theorem~\ref{thm:4} shows that there are four isomorphism classes of amalgams of this type.
We let $\mathcal U_4=\{G_1, G_2\}$ and  view $\mathrm{AGL}_4(2)$ as a subgroup of $\mathrm L_5(2)$. We then have that
\begin{eqnarray*}
G_1 & = &  \la a_1,a_2,a_3,a_4,a_5,a_6,a_7,a_8,a_9,a_{10},a_{11},a_{12}\ra, \\
G_2 & = &  \la a_1,a_2,a_3,a_4,a_5,a_6,a_7,a_8,a_9,a_{10},a_{11},a_{13}\ra, \\
G_1 \cap G_2 & = & \la a_1,a_2,a_3,a_4,a_5,a_6,a_7,a_8,a_9,a_{10},a_{11}\ra
\end{eqnarray*}
where $a_i$ is the element of $\mathrm{L}_5(2)$ with 1's on the diagonal and 0's everywhere except for the position of $a_i$ given below.
$$\left( \begin{array}{ccccc}
         1   & 0   & 0      & 0      & 0        \\
         a_1 & 1   & a_{11} & 0      & 0      \\
         a_2 & a_3 & 1      & a_{12} & 0 \\
         a_4 & a_5 & a_6    & 1      & a_{13}  \\
         a_7 & a_8 & a_9    & a_{10} & 1       
         \end{array}\right) $$

To obtain the amalgams of the same type as $\mathcal U_4$ in the different isomorphism classes we need to use the automorphisms $\alpha$ (see (\ref{alpha def})) and $\beta$ (see (\ref{beta def})) of $G_1 \cap G_2$ given below (where $\alpha$ and $\beta$ act trivially on the generators not listed).

\begin{eqnarray*}
\alpha & : &  \left \{ \begin{array}{ccc} a_{3} &\mapsto &a_7 a_{3}\\ a_{11} & \mapsto& a_7 a_{11} \end{array}\right . \\
\beta & : &  \left \{\begin{array}{ccc} a_1 &\mapsto &a_1a_7 a_9\\ a_2 & \mapsto & a_2a_7a_8 \end{array}\right.
\end{eqnarray*}

Let us write $\mathcal U_4^\sigma$ for the amalgam obtained from $\mathcal U_4$ using the map $\sigma \in \{\alpha,\beta,\alpha\beta\}$. Note that $E_1:=\la a_7 \ra$, $E_2:=\la a_1,a_2,a_4,a_7\ra$ and $E_3 = \la a_7,a_8,a_9,a_{10} \ra$ are the only normal subgroups of $G_1$ contained in $G_1 \cap G_2$. The amalgams $\mathcal U_4$ and $\mathcal U_4^\alpha$ are unfaithful precisely because $E_2$ is normalised by $G_1$, $G_2$ and by $\alpha$. On the other hand $E_1$ and $E_3$ are normalised by $\beta$, but not by $a_{13}$, and $[ E_2^\beta, a_{13}] \nleq E_2^\beta$. Thus the amalgams $\mathcal U_4^{\alpha\beta}$ and $\mathcal U_4^{\beta}$ are faithful.

For $\sigma \in \{\beta, \alpha\beta\}$ we denote the universal completion of the amalgam $\mathcal U_4^\sigma$ by $\mathcal G^{\sigma}$. For $1\leqslant i \leqslant j \leqslant 13$ we let $R(i,j)$ be a relation between $a_i $ and $a_j$  that holds in $\mathrm L_5(2)$ and for $1 \leqslant i \leqslant 11$ we write $R^\sigma(i,j)$ for a relation between $a_i^\sigma$ and $a_j$. Then we obtain
%We calculate the set $\mathcal R$ of relations between the elements $a_i$ and $a_j$ for $i,j\in [1,12]$ and the set $R^\sigma$ of relations between between $a_{13}$ and $a_j^{\sigma}$ for $j\in [1,11]$. Then, with $\mathcal R$ and $\mathcal R^\sigma$ as below, we have
$$\mathcal G^\sigma =\left \langle \begin{array}{c|c} a_1,a_2,a_3,a_4,a_5,a_6,a_7,

& R(i,i) \text{ for } 1\leqslant i \leqslant 13,\ R(i,j)\text{ for } 1 \leqslant i < j \leqslant 12,\\ 
a_8,a_9,a_{10},a_{11},a_{12},a_{13}&
  R^\sigma(i,13)\text{ for } 1 \leqslant i \leqslant 11 \end{array} \right \ra.$$
For the relations we have $R(i,i)=a_i^2$ for $1 \leqslant i \leqslant 13$, 
$R(3,11)=(a_3a_{11})^3$, 
$R(10,11)=(a_{10}a_{11})^3$, 
$R(6,12)=(a_6 a_{12})^3$, 
$R(11,12)=(a_{11}a_{12})^4$, and the remaining relations are of the form  $R(i,j)=[a_i,a_j]w(i,j)$ for some $w(i,j) \in G_1 \cap G_2$ which can be calculated by directly multiplying the matrices above. We note explicitly that $R^\beta(1,13)=[a_1,a_{13}]a_4a_6$, $R^\beta(2,13)=[a_2,a_{13}]a_4a_5$, $R^{\alpha\beta}(3,13)=[a_3,a_{13}]a_4$ and $R^{\alpha\beta}(11,13)=[a_{11},a_{13}]a_4$.

%
%\begin{center}
%\begin{tabular}{c | c }
%$\mathcal R$ & $a_i^2 = 1$ for $i \in [1,13]$, $[a_1,a_2]=[a_1,a_3]=[a_1,a_4]=[a_2,a_3]=[a_2,a_4]=[a_3,a_4]=1$ \\ \hline
%$\mathcal R^\beta$ & some relations \\ \hline
%$\mathcal R^{\alpha\beta}$ & some slightly different relations
%\end{tabular}
%\end{center}

We observe that the subgroup $L=\la a_3,a_5,a_6,a_8,a_9,a_{10},a_{11},a_{12},a_{13}\ra $ of $\mathcal G^{\beta}$ is complemented by the elementary abelian subgroup $\la a_1,a_2,a_4,a_7\ra$ of order $2^4$. This gives a representation $\pi$ of $\mathcal G^\beta$ of degree $16$. Moreover, since each normal subgroup of $G_1$ contains $a_7$ and each normal subgroup of $G_2$ contains the subgroup $\la a_4,a_7 \ra$, $\pi$ restricted to $G_1$ or $G_2$ is faithful. We claim that $A_{16}$ is a faithful completion of $\mathcal G ^\beta$, that is, $\mathcal G^\beta \pi= A_{16}$. Since $G_1$ is perfect, we have $G_1 \pi \leqslant A_{16}$ and it is easy to check that $a_{13}\pi$ is of cycle type $(a,b)(c,d)(e,f)(g,h)$. Thus $\mathcal G^\beta\pi $ is an insoluble transitive subgroup of $A_{16}$ so if the claim is false it must be that $\mathcal G^\beta\pi$ is contained in a transitive insoluble maximal subgroup, conjugate to one of 
$$\mathrm{AGL}_4(2),\ S_2 \wr S_8,\ S_8\wr S_2.$$
We note that $G_2 \pi$ preserves only blocks of size four by examining the action on the regular normal subgroups of order $2^4$. If  $G_i\pi \leqslant \mathrm{AGL}_4(2)$ for $i\in \{1,2\}$ then order considerations show that $G_i\pi$ contains a Sylow 2-subgroup of $\mathrm{AGL}_4(2)$ and therefore contains the regular normal subgroup  of order $2^4$. Since $\mathcal U_4^\beta$ is faithful therefore, we have $\mathcal G^\beta\pi \nleqslant \mathrm{AGL}_4(2)$ and this gives the claim.

In the introduction we noted that $M_{24}$ and $He$ are completions of faithful amalgams of type $\mathcal U_4$. Indeed, using {\sc Magma} \cite{magma} we see that adding the following set of relations to $\mathcal G_4^{\alpha\beta}$ gives $M_{24}$ as a quotient:
$$\{   (a_6  a_{12}  a_{13})^5,
 (a_{11}  a_{12}  a_{13})^{11},
 (a_{10}  a_{12}  a_{13})^5   \}$$
and adding the following set of relations gives $He$ as a quotient:
$$ \{  (a_{12}  a_2  a_8  a_{13})^5 ,\ 
(a_6  a_{12}  a_2  a_7  a_8  a_{13})^5 ,\ 
(a_{10}  a_8  a_{13}  a_{12}  a_7)^5 \} .$$
In the same way we find  $\mathcal G_4^\beta$ has no index 24 subgroup and that  $\mathcal G_4^{\alpha\beta}$ has no index 16 subgroup. Thus $A_{16}$ is a completion of  $\mathcal U_4^{ \beta}$ alone and the sporadic groups $M_{24}$ and $He$ are completions of  $\mathcal U_4^{\alpha \beta}$ only.

% 
%-  - index 24 subgroup of $G^{\beta\alpha}$
% 
% - explanation of how this arises internally in $M_{24}$.

\end{document}